\documentclass[A4,14pt]{article}

\usepackage{etoolbox}
 \usepackage{ifthen}
\usepackage{makecell}
\usepackage{booktabs}
\usepackage{rotating}
\usepackage{thumbpdf,lmodern}
 \usepackage{bigints}
\usepackage[hidelinks, pdfpagemode={UseNone}, pdfstartview={XYZ null null 1.0}, pdfview={XYZ null null 1.0}]{hyperref}
\usepackage{subcaption}
\usepackage{relsize}
\usepackage{multirow}
\usepackage{minitoc}
 \usepackage{pdfpages}
\usepackage{amsmath,amsfonts,amssymb,amsthm,mathrsfs}
\usepackage{stmaryrd}\usepackage{latexsym}
\usepackage{graphicx}
\usepackage{verbatim}
\usepackage{ulem}
\usepackage{tikz}
\usepackage{environ}
 \usepackage[latin1]{inputenc}  
\usepackage[T1]{fontenc}  
\usepackage{comment}
\usepackage{cleveref}

 \newcommand{\cancel}[1]{}
 \newcommand{\Coul}{\text{\rm{Coul}}}
\newcommand{\rr}{\mathfrak{r}} 

 \newcommand{\fluct}{\text{\rm{fluct}}}  
 \newcommand{\X}{     \mathsf X}  
    
 \renewcommand{\k}{{  \bf k}}  
    
 \newcommand{\rev} {
 }
 \newcommand{\s}{   \mathsf s}

\newcounter{appx}
\renewcommand{\theappx}{\Alph{appx}}

\newcommand{\appsection}[2]{%
  \refstepcounter{appx}%
  \section*{Appendix \theappx: #1}\label{#2}%
  \addcontentsline{toc}{section}{Appendix \theappx: #1}%
}

\author{Rapha\"el Lachi\`eze-Rey}

  \newcommand{\what}[2]{ {\widehat{ #1}}^{ \hspace{-.1cm}    ~^{#2}}    }

\usepackage{color}

\newcommand{\revPMP}{}

 \newcommand{\sinc}{ \text{\rm{\color{black}sinc}}}  
 \newcommand{\F}{  \mathscr  F}  
 \usepackage{geometry}
\newtheoremstyle{myexample}     
  {3pt}                         
  {3pt}                         
  {}                            
  {}                            
  {\bfseries}                   
  {.}                           
  { }                           
  {}                            
 
 \newcommand{\Z}{\mathbb{Z} _{\lambda }^{ d}}

 \newcommand{\Leb}[1][d]{  \mathscr  L^{ #1}}

 \newcommand{\T}[1][\lambda]{   \mathbb  T  _{ #1}^{ d}}  
\theoremstyle{myexample}

\newtheorem{theorem}{Theorem}[section]
\newtheorem{lemma}{Lemma}[section]
\newtheorem{proposition}{Proposition}[section]
\newtheorem{remark}{Remark}[section]

\newtheorem{corollary}{Corollary}[section]
\newtheorem{assumption}{Assumption}[section]

  \newcommand{\maxscale}{ b}

\renewcommand{\P}{\mathsf{P}}
\newcommand{\M}{\mathsf{M}}
\renewcommand{\S}{\mathsf{S}}
\newcommand{\C}{\mathscr{C}}

\newcommand{\B}{\mathscr{B}}

 \newcommand{\N}{  \mathscr  N}  
  \newcommand{\equlaw}{ \stackrel{(d)}{ = }}

\author{ Raphael Lachieze-Rey\thanks{INRIA Paris, Team Mathnet, and Lab. MAP5, Université Paris Cité, raphael.lachieze-rey@math.cnrs.fr}}

\title{From smooth to discontinuous kernels: a variance transfer principle for hyperuniform processes}
\begin{document} 
\maketitle

 \begin{quote}

 {\bf Abstract.} We give a transfer principle for   the fluctuations of linear statistics of finite particle systems around Lebesgue measure: if  for a smooth kernel  the variance decays polynomially with some exponent $ \alpha \geqslant 0$ compared to   independent (non-interacting) particles, then the number variance over balls decays on average with exponent $ \min(\alpha ,1)$, times a $ \log$ term if $ \alpha  = 1$, over a possibly reduced range of scales for non-periodic systems. Such a behaviour is called $ \alpha $-hyperuniformity.
 
 We apply this principle to eigenvalues of random $ N\times N$ Girko matrices, leveraging results of Cipolloni et al. \cite{CES}, and obtain the optimal perimeter-like number variance,  {\revPMP on some mesoscopic  scales}, after local averaging. We also apply the results to  Coulomb gases, by transferring the results of Serfaty \cite{Serfaty-HU}: we prove that $ 2D$-Coulomb gases are $ 2$-hyperuniform, i.e. they have surface order number variance, which is optimal, and that $ 3D$ Coulomb gases are $ 1$-hyperuniform. In $ 2D$, it allows to prove finite regularised Coulomb energy and finite Wasserstein distance to Lebesgue measure, and have a sharp lower bound on a constant sign decorrelation exponent. 
 \end{quote}
 
  {\bf keywords:} point processes, hyperuniformity,  Coulomb gases, Gibbs measures, random matrices, Girko random matrices.\\
 
  {\bf MSC classification:} 60G55 (Point processes),   82B05 (Classical equilibrium statistical mechanics), 60B20 (Random matrices), 60K35  (Interacting random processes; statistical mechanics type models), 60G57 (Random measures), 60D05 (Geometric probability and stochastic geometry)
 
 \tableofcontents
 \label{toc}
 
  \section{Introduction}

Let $N\geqslant 1,\X_{ N}$ be a random finite point configuration in $\mathbb R^d$  and for an integrable kernel $ f:\mathbb{R}^{ d} \to  \mathbb C $, let
\[
\X_{ N}(f):=\sum_{x\in \mathcal \X_{ N}} f(x)
\]
denote the associated linear statistic. Linear statistics are among the basic observables in the study of interacting  particle systems such as random matrix ensembles, Gibbs measures, or zeros of random functions. They interpolate between smooth mesoscopic observables, obtained from regular test functions, and counting observables, obtained from indicators of sets. A central example is the number of particles in the ball $ B(x,R)$ with center $ x$ and radius $ R>0$,
\[
\X_{ N}(B(x,R)):=\# \X_{ N}\cap {B(x,R)},
\]
whose variance is also called the \textit{number variance} (for balls). 
We study linear statistics $ \X_{N }(f)$, for a kernel $ f$ smooth or irregular, often  in the space $ \B_{ c}(\Sigma )$ of bounded functions compactly supported within some set $ \Sigma \subset \mathbb{R}^{ d}.$

Define $ f_{ R  } = f(\cdot /R)$ for some $ R = R_{ N}\geqslant 1$ called the  {\it scale}.   The canonical model without interaction is the  Poisson model with intensity $ 1$ on $ \Sigma  = B(0,N^{ 1/d})$. In this case, we have for  $ f\in \B_{ c}(B(0,1)) $ and $ R\leqslant N^{ 1/d}$, as $ N\to \infty ,$
\begin{align*}
 \mathbf E \left[
\X_{ N}(f_{ R})
\right] =&  \int_{ \Sigma  }f_{ R } \sim R^{ d}\int_{  }f ,\\
  \textrm{Var}\left(\X_{ N}(f_{R })\right) = & \int_{\Sigma  } | f_{ R } | ^{ 2} \sim R^{ d}\int_{  } | f | ^{ 2}.
\end{align*}
Similar rates are obtained if $ \X_{ N}$ consists of    i.i.d. points uniform on $ \Sigma .$
In particular, the variance has  the same order as the expected number of points in $ f_{ R }$'s support. 

\subsection{Hyperuniformity}
  
 A major theme in the theory of interacting particle systems is the phenomenon of \textit{hyperuniformity}, namely the anomalous suppression of large-scale density fluctuations. Hyperuniformity at some scale $ R_{ N}\in [1,N^{ 1/d}]$ going to $ \infty $ means that the number variance in balls is reduced,
\[
\mathrm{Var}(\X_{ N}(B(0,R_{ N}))) = o(R_{ N}^{ d})  \qquad (N\to\infty).
\]

For several models in random matrices or Gibbs measures, the traditional scaling is to consider $ \tilde \X_{ N}: = N^{ -1/d}\X_{ N}$, in which case the points tend to concentrate around a fixed compact $ \Sigma _{ 0}$ and the expected variance for $ \tilde \X_{ N}(f_{ \rho })$ is of magnitude $(N^{ 1/d}\rho )^{ d-\alpha } $ in this compact. We stick to the blowup  {\it microscopic} formulation with $ \X_{ N}$ as it is more natural in view of the methods employed here, and it also leads to a global notational simplification; it is also the natural scale to study infinite volume limits $ \X: = \lim_{ N}\X_{ N}$ in the space of point processes.

Reduced charge fluctuations in certain models, also referred to as {\it sum rules} in physics, were already observed  or conjectured a long time ago for particle systems \cite{martin1980charge, Leb}, or for zeros of random polynomials and eigenvalues of random matrices  \cite{ForresterHonner,Forrester}. 
The systematic investigation of this phenomenon  started in the 2000's with the team of S. Torquato  at Princeton  \cite{TS03}, who popularized the term  {\it hyperuniformity,} or J. Lebowitz at Rutgers University, sometimes under the terminology of {\it superhomogeneity}. Part of the interest from condensed matter physicists stems from the interpretation of hyperuniformity as transparency to long wavelengths.
 Besides its usefulness and appearances in other sciences, many popular mathematical models turned out to be  hyperuniform, in random matrices, statistical physics, random polynomials, quasicrystals, generating a literature growing steadily, see the surveys  \cite{Tor18,GosLeb} for an in depth collection in physics or statistical physics, or the  recent mathematical survey  \cite{survey-hu}. 

\subsection{The transfer principle}  

 In the vast majority of cases, establishing sharp variance bounds for $ \X_{N }(f_{ R_{ N}})$ through mathematical analysis is easier for $ f$ smooth.
This is notably the case for the examples that we explore here,  where Fourier or energy methods  yield variance bounds for regular test functions, while the physically more visible number variance corresponds to a discontinuous kernel.  A recurring question is whether variance bounds obtained for smooth kernels are valid for discontinuous observables such as counting functions of geometric sets $ f = 1_{ B}$, e.g.\cite{CES,BauerschmiditBourgadeNikulaYau,Serfaty-HU}.  The main purpose of this paper is to show that  fluctuation bounds around Lebesgue measure for one fixed smooth non-negative test function automatically transfer to a broad class of irregular observables including ball indicators, after mandatory exponent truncation, in a broad framework.

 We more precisely assume that fluctuations around Lebesgue measure satisfy for some non-degenerate kernel $ f\geqslant 0$ with compact support, and all $ R$ in some scales range $ [1,N^{ \maxscale }]\subset [1,N^{ 1/d}]$,
\begin{align*}
 \mathbf E\left[
 \left|
\X_{ N}(f_{ R }(x + \cdot ))-R^{ d}\Leb(f)
\right| ^{ 2}
\right]\leqslant CR^{ d-\alpha }
\end{align*}  for all $ x$  in the bulk, i.e. not too close from $ \partial \Sigma ,$ where $ \Leb$ denotes Lebesgue measure, and $ \alpha \geqslant 0$ is the  {\it  hyperuniformity index} (there is no hyperuniformity if $ \alpha  = 0$). Then we deduce that for most points    $ x_{ N} $ in the bulk, for $ R\in [1,N^{ \maxscale }]$ 
\begin{align*}
  \textrm{Var}\left(\X_{ N}(B(x_{ N},R))\right) = O(R^{ d-\min(\alpha ,1)})\ln(R)^{ {\bf 1}\{ \alpha  = 1 \}},
\end{align*}
see the precise statement at Theorem \ref{thm:non-asymp}. We also state in  Theorem \ref{thm:main-asymp} an infinite volume version:  for a system $ \P$ which is the limit of a subsequence $ \tau _{ x_{ N_{ j}}}\X_{ N_{ j}},N_{ j}\to \infty ,$ in the space of random measures,   the growth rate of $  \textrm{Var}\left(\P(B(x,R))\right)$ as $ R\to \infty $ is in average over $ x$ in $ O(R^{d -\min(\alpha ,1)})\ln(R)^{ {\bf 1}\{ \alpha  = 1 \}}$. When $ \P$ is stationary, the rate is the same at all points, it means rigorous $ \alpha $-hyperuniformity for $ \P$.

The ideal case, where the transfer principle applies flawlessly, is 
when the system is periodic, i.e. defined on the torus $ \mathbb{R}^{ d}/ (N^{ 1/d}\mathbb{Z} )^{ d}$ with periodic boundary conditions. When this is not the case,    $ \maxscale $ must be smaller than $ (d(d + 2))^{ -1}$.  This reduced scale range  instead of $ \maxscale  = d^{ -1}$ is an artifact of the  current method, where we extract the approximately stationary local restrictions of the model and compare them with   stationarized versions.  

Also, the obtained rates apply to a broad range of irregular kernels, not only ball indicators. For $ \gamma >0,$ let $ \F_{ \gamma }$ be the class of bounded and compactly supported functions $ \varphi $ such that the Fourier transform $ \hat \varphi $ satisfies $$  | \hat \varphi (\xi ) |  = O(\|\xi \|^{ -\frac{ d + \gamma }{2}}),\quad  \xi \to \infty .$$ The unit ball indicator $ \varphi  = 1_{ B_{ 1}}$ belongs to $ \F_{ 1}$ (see  \eqref{eq:bessel}). In the case of general $ \gamma >0$, the truncated variance rate is rather $$O( R^{ d-\min(\alpha ,\gamma )}\ln(R)^{ {\bf 1}\{ \alpha  = \gamma  \}}).$$

Let us discuss the obtained variance rates:\begin{itemize}
 \item For $ \alpha  = 0,  \textrm{Var}\left(\X_{ N}(B(x_{ N},R))\right) \asymp R^{ d}$  corresponds to the the rate of   i.i.d variables uniform in $  \Sigma $,  or   to a Poisson process, or any other ``standard'' process (zeros of random stationary Gaussian functions, cluster or Cox processes, short range Gibbs measures, ...)
\item The case $ \alpha >0$ implies a variance reduction, or main term cancellation, i.e.  {\it hyperuniformity}. More precisely, the assumption means that the system is at least ``$ \alpha $-hyperuniform''. For $ \alpha \in (0,1)$, the exponent is not truncated, this is the case for instance for Riesz gases with exponent $ s\in (0,1)$, where $ \alpha  = 1-s$, see Section \ref{sec:riesz}.
\item We see that a transition occurs at $ \alpha  = 1$.   The $ \ln(R)$ term for $ \alpha  = 1$, corresponding to ``$ 1$-hyperuniformity''  is sharp, in the sense that many models for which $   1 $ is the largest possible exponent whatever is $ f$, such as the GUE ensemble in dimension $ d = 1$, have a logarithmic number variance. The case $ \alpha >1$ in 1D corresponds to a bounded number variance, it is the case for instance for the one-dimensional Coulomb gas \cite{GosLeb}.
\item  Torquato \cite{Tor18} divides hyperuniform systems in three classes $ \alpha >1$  (class I)$,\alpha  = 1$  (class II)$,\alpha \in (0,1)$ (class III), reflecting number variance behavior over growing balls. It is also formally possible to build  hyperuniform point processes falling in none of these classes, i.e.  for which the variance reduction is sub-polynomial  \cite{survey-hu,DFHL}.
  \item For $ \alpha >1$ (class I), the balls number variance growth is in surface order $ R^{ d-1}$. By Beck's lower bound principle \cite{Beck87}, this is a universal lower bound in $ \limsup$ (see for instance  \cite[Thm. 2.2]{survey-hu}).
The precise value of $ \alpha $ is therefore uninformative for the   number variance over balls. It does not mean that all class I systems are similar, they encompass both periodic and disordered systems. They can in particular behave very differently when investigated over smooth kernels, especially for large $ \alpha $, where it can be useful in numerical integration.
  \item The most standard case is $ \alpha  = 2$, corresponding to a sufficiently disordered  hyperuniform model, or more formally to a sufficiently decaying covariance function (\cite{GosLeb,survey-hu}), such as for the 1D Coulomb gas or the 2D Ginibre ensemble.

\item   For the variance of irregular linear statistics others  than the ball indicator, in particular for $\varphi \in \F_{ \gamma }, \gamma \neq 1$, a similar transition occurs at $ \alpha  = \gamma $ under some mild conditions. This leaves the possibility of kernels with a square integrable singularity (see Section \ref{sec:periodic}).

\item {We also have converse bounds, in the sense that number variance cannot decay faster than variance of smooth linear statistics {\rev in the stationary case}, see Theorem  {\rev \ref{thm:main-periodic}}. In any  case, for most test functions $ \varphi \notin \F_{ \gamma },$ one cannot have a rate $ O(R^{ d-\gamma })$, using  a variant of   Beck's argument.}

\item While most known models have index $ \alpha  = 2$, or $ \alpha  = 1$ (more often in dimension $ d = 1$), there exists a unique natural disordered model for which it is proven that $ \alpha >2$; it is the zero set of Weyl polynomials, or in the infinite volume limit the zero set of the planar Gaussian Analytic Function, for which $ \alpha  = 4$  \cite{ST06}. There also exist some constructions yielding arbitrarily high $ \alpha $, see  \cite{torquato-tilings}.

\item The purely atomic nature of particle systems is not used, and results should apply to other classes of random measures.
\end{itemize}

  \subsection{Applications to random matrices and Gibbs measures}  
 
We apply this principle to models of probability theory and statistical physics. Coulomb gases, also called  {\it one-component-plasmas}, or  {\it jelliums,} form a class of models where particles  interact pairwise via the Coulomb potential. It is a major theme in the statistical physics literature and is connected to many applications in mathematics and physics. See for instance the surveys  \cite{Lewin,serfaty-book} and references therein. Serfaty \cite{Serfaty-HU} and Bauerschmidt et al.  \cite{BauerschmiditBourgadeNikulaYau} proved variance cancellation for smooth linear statistics, with optimal rate in 2D,  and Serfaty writes that  {\it treating the case of non-smooth kernels (...) remains a (significantly more) delicate problem.} Leblé  \cite{Leble} rigorously proved hyperuniformity for the 2D Coulomb gas    $ \X_{ N}$ with quadratic confinement with a logarithmic variance reduction, i.e. for some constant $ C,$ 
\begin{align}
\label{eq:leble-bound}
 \textrm{Var}\left(\X_{ N}(B(x,R))\right) \leqslant CR^{ 2}\ln(R)^{ -0.6},
\end{align} for $ B(x,R)$   in the bulk. 
 {With the current transfer principle, we are able to prove  hyperuniformity with optimal surface order rate $ O(R) $ in $ 2D$, and surface order magnitude times a logarithmic factor $ O(R^{ 2}\ln(R))$ in 3D, with the aforementioned caveats of reduced scales range and local averaging in the finite setting, see Theorem  \ref{thm:3D-Coul-finite}. As mentioned earlier, obtaining the full range $ [1,N^{ 1/d}]$ could be possible by working with the periodic version of the Coulomb gas, as in  \cite{CohHer}, or the 1D application to Riesz gases at Section \ref{sec:riesz}. These caveats also disappear when studying a limit point $ \P$ in the space of random measures, or  {\it infinite volume Coulomb gas}, see Theorem \ref{thm:Coul-infinite}.
This proves that they are resp. $ 2-$ and $ 1-$hyperuniform. In dimension $ 2$, this allows thanks to the recent results of Huesmann and Leblé  \cite{HueLebl} to prove finite Coulomb energy and Wasserstein distance per unit volume, which was not yet allowed by    \eqref{eq:leble-bound}, see Theorem \ref{thm:spectral-2D-Coul}. }

Our second application concerns eigenvalues of random matrices.  We  consider here $ N\times N$ Girko random matrices, i.e. matrices filled with  i.i.d.  entries with finite second moment. The most famous example is Ginibre ensemble, obtained when the entries are Gaussian, where reduced variance in balls is known since Ginibre  \cite{Ginibre}. One of the major line of research in random matrix theory is  {\it universality}, i.e. proving that qualitative results are valid beyond the Gaussian case, under sole moment or density assumptions on the entries, see for instance the monographs  \cite{GuionnetBook,Mehta}.    Recently, Cipolloni, Erd\"os and Schr\"oder  \cite{CES} showed hyperuniformity for $ H_{ 0}^{ 2}$-smooth kernels, and commented that  {\it extending this study to less regular $ f$'s creates substantial difficulties}. Later, Cipolloni, Erd\"os and Kolupaiev \cite{CipoErdosKol} showed  hyperuniformity on balls $ B(x,R)$, with rates of the form $ R^{ 2-\varepsilon }$ at all non-microscopic scales for   $ \varepsilon  = 1/20$. By transferring the rates of \cite{CES} about smooth linear statistics, we are able to give   rates in $ R^{ 1   {\revPMP  + o(1)} }$    for the balls number variance  {\revPMP at scales $ R^{ b},b \in (0,{ 1/8})$}, except on a vanishing set of centers.  

These two applications are direct consequences of the main result Theorem \ref{thm:non-asymp}, valid for general particle systems, and the reduced scales range $ [1,N^{   \frac{ 1}{d(d + 2)}}]$ is a consequence of some boundary effect that we explain below. 

\subsection{Proof ideas and plan}  

The fundamental result, core of this paper, is Theorem \ref{thm:abstract-general}, a variance transfer principle for any  particle system on the torus $ \mathbb{R} ^{ d}/ \mathbb{Z}^{ d} $ which law is invariant under torus translations. Typical examples are Gibbs measures with periodic boundary conditions  \cite{boursier,CohHer}. In this case, the transfer principle applies flawlessly, without boundary effect, reduced scales range, or local averaging. We illustrate this concept by applying it to periodic $ s$-Riesz gases on $ \mathbb{R}/\mathbb{Z} $, leveraging the results of Boursier \cite{boursier}; we are able to treat kernels with a square integrable singularity, regardless of $ s\in (0,1)$, at all scales. 
The proof is essentially spectral, relying on the finite-volume spectral measure on the dual lattice of the torus and decompose the variance into low and high frequency contributions. The low-frequency part is controlled by the assumed bound for the smooth statistic, while the high-frequency part is estimated using only the Fourier decay of the target kernel and the uniform translation boundedness of the spectral measure. 

 For general non-invariant $ \X_{ N}$, we study a subsample of $ \X_{ N}$ restricted on a small cube of the form $ x_{ 0} + [1,N^{ 1/d }\varepsilon _{ N}]^{ d}$ where $ \varepsilon _{ N}\to 0$ and the process restriction is supposed to be asymptotically homogeneous. Hence this restriction can be approximated by a purely invariant model, for which the periodic variance transfer principle applies, as explained above. Using the approximation in the opposite direction allows to give similar number variance rates for the initial restriction. Boundary terms arising from the restriction  produce higher-order fluctuations, which forces the condition $ \maxscale <1/d$ in order to not kill the rates obtained with the transfer principle.\\

The paper is organized as follows. In Section  \ref{sec:main} we state the transfer principle for finite size systems, with and without exact stationarity, and the consequences for infinite volume limit points. Sections \ref{sec:girko} and  \ref{sec:gibbs} are devoted to applications to respectively random matrices and Gibbs measures. In Section  \ref{sec:prfs} we give the general result on the torus and the proofs of Section  \ref{sec:main} results.

   \section{Transfer principle for particle systems}
   \label{sec:main}
Let $N\geqslant 1$  and $ \X_{ N}$ a locally finite random set of points in $ \mathbb{R}^{ d}$, and some  domain $ \Sigma = N^{ \frac{ 1}{d}}\Sigma _{ 0}$ where $ \Sigma_{ 0} \subset \mathbb{R}^{ d}$ is fixed.  
For $ A\subset \mathbb{R}^{ d}, k\in \mathbb{N}\cup \{\infty \}$, denote by $ \C_{ c}^{ k}(A)$ the class of continuously $ k$-times differentiable functions with compact support in $ A$, and $ \B_{ c}^{}(A)$ the class of bounded functions with compact support contained in $ A$, endowed with $ \|f\|_{ \infty } = \sup_{ x\in A} | f(x) | $. 
 {Denote the Fourier transform on $L^{ 1}( \mathbb{R}^{ d})$ with 
\begin{align*}
 \hat \varphi (\xi) = \int_{ \mathbb{R}^{ d}}e^{ -i\xi\cdot x}\varphi (x)dx,\xi\in  \mathbb{R}^{ d}.
\end{align*}  For $ A\subset \mathbb{R}^{ d}$, denote   by $\F_{ \gamma }(A) $  the space of   $ \varphi \in \F_{ \gamma }$ supported by $ A$. Also define for $ R>0$,
\begin{align*}
 B_{ R}: = B(0,R).
\end{align*}We have \begin{align*}
\widehat{ 1_{ B_{ 1}}}(\xi) =(2\pi )^{ d/2} \|\xi\|^{ -d/2}J_{ d/2}(\|\xi\|)
\end{align*}
where $ J_{ d/2}$ is the Bessel function of order $ d/2$, see for instance  \cite{SteinWeiss}. It implies in particular 
\begin{align}
\label{eq:bessel}
 |  \widehat{ 1_{ B_{ 1}}}(\|\xi\|) |   = O((1 + \|\xi\|)^{ -\frac{ d + 1}{2}}),
\end{align} i.e. $ \varphi  = 1_{ B _{ 1}}$ belongs to $ \F_{ 1}$ in any dimension. }

We study the fluctuations of linear statistics around Lebesgue measure, i.e. for $ f\in  \B_{ c}^{}(\mathbb{R}^{ d} )$, let  {  $$  \overline{\X_{ N}}(f):= \sum_{x\in \X_{ N}}f(x)-\Leb(f).$$}

For $ x\in \mathbb{R}^{ d}$, denote by $ \tau _{ x}:y  \mapsto  x + y$ the translation operator, that naturally extends to sets and functions over $ \mathbb{R}^{ d}$, with $ \tau _{ x}f(y) = f(y-x),\tau _{ x}A = \tau _{ x}1_{ A} = \{y + x;y\in A\}.$
Define the centered cube with sidelength $ L>0$ $$ \T[L ] = \left[
-L /2,L/2
\right)^{ d},$$
and for $R >0,$ the $ R$-interior of $ \Sigma :$
$$ \Sigma  ^{ R  } := \{x: \tau _{ x}  \mathbb T^{ d}_{2R } \subset \Sigma\} .$$\\

 \subsection{Non-asymptotic result}  
 \label{sec:non-asymp}
 
    Define for $ x_{ 0}\in \mathbb{R}^{ d},\varepsilon >0$ the $ \varepsilon -$local  variances average  : for $ f\in  \mathscr  B_{ c}(\mathbb{R}^{ d}),$  
\begin{align*} 
\overline{\textrm{Var}}^{\varepsilon  }\left(\X_{ N}(f)\right):
 = \frac{1}{(\varepsilon N^{1/d})^{d}}\int_{\T[\varepsilon N^{1/d}]}  \textrm{Var}\left(\X_{N}(\tau _{  x}f )\right)dx.
\end{align*}
Say a function $ f$ is  {\it non-degenerate} if on a non-empty open set $  | f | $ is lower bounded by some $ c>0.$ Let $  \mathscr  B_{ c}^{  + }(K)$ the class of non-negative non-degenerate functions defined on $ K$. Define  for $ R>0,\alpha \in \mathbb{R},\gamma >0$ the target rate 
\begin{align*}
 V_{ \alpha ,\gamma }(R) : = R^{ d-\min(\alpha ,\gamma )}\ln(2 + R)^{ {\bf 1}\{ \alpha  = \gamma  \}}.
\end{align*}

 { \begin{theorem} \label{thm:non-asymp}(Non-asymptotic result) Let $ d\geqslant 1,  \alpha \geqslant 0.$ Let $\maxscale,\delta >0$  such that $\maxscale (d + 2) + \delta <\frac{ 1}{d}$. 
 Assume that there exists   $ f\in \B_{ c}^{ +  }(B_{ 1})$  such that for some finite constant $ C_{ f}$,    for all $N\geqslant 1,R \in [1,N^{ \maxscale }]  ,$  
\begin{align}
\label{eq:ass-asymp-stat}
\sup_{ x\in \Sigma ^{ R }}\mathbf E  \left[\;
   \overline{\X_{ N}}(\tau _{ x}f_{ R  }) ^{ 2}
\right]\leqslant C _{ f}  R ^{ d  -\alpha }.
\end{align}

\begin{itemize}
\item  {\revPMP (i)} Then   for $N\geqslant 1,\varepsilon _{ N}: = N^{ -\delta },$
  for $\gamma \in (0,1], \varphi \in \F_{ \gamma }(B_{ 1})$ we have  for all $ R \in [R^{ \circ}_{ f},N^{ \maxscale }],$   
\begin{align}
\label{eq:result-main-non-asymp}
\sup_{ x\in \Sigma ^{2\varepsilon _{ N}N^{ 1/d}  }} \overline{  \textrm{Var} }^{ \varepsilon _{ N}}\left[
{ \X_{ N}}(\tau_{ x }\varphi _{  R} )
\right]\leqslant C_{ \varphi ,f} V_{ \alpha ,\gamma }(R)\end{align}

where $ C_{ \varphi ,f},R^{ \circ}_{ f}<\infty $ do not depend on $ N,R.$\item 
 {\revPMP (ii) If for arbitrarily small $ \varepsilon >0,$  \eqref{eq:ass-asymp-stat}  only holds for $ R\in [N^{ \varepsilon },N^{ b}]$, 
  then  \eqref{eq:result-main-non-asymp} holds for $ \varepsilon $ sufficiently small with $ C_{ \varphi ,f,\varepsilon }V_{ \alpha  ,\gamma -\sqrt{\varepsilon} }(R)$ on the RHS for $ R\in [N^{ q\sqrt{\varepsilon }},N^{ b}]$ where $ q = 2d + (1\vee (\gamma -d))$.
  }
\end{itemize}

 
\end{theorem}}
   As we show in   Theorem  \ref{thm:main-asymp} below, the bound involving  local averaging implies that for most points, asymptotically the number  variance has magnitude $O (V_{ \alpha ,\gamma }(R))$.

The proof is at Section  \ref{prf:main-non-asymp}.
We illustrate  {\revPMP point (ii)} at Section \ref{sec:girko}   with the  Girko random matrices eigenvalues in dimension $ d = 2$, for which previous results  \cite{CES} about the variance of smooth linear statistics provide a uniform bound, i.e. $ \alpha  =  2$, at all scales  {\revPMP except microscopic (i.e. as in point (ii))}. The averaged number variance over a ball $ B(x_{ 0},R)$ in the bulk  is therefore in $ O(R^{  {1  {\revPMP  + o(1)} }}) $, i.e. $
O  \left(
\mathbf E\left[
 \X_{ N}(B(x_{ 0} ,R ))
\right]
\right)
^{   \frac{ d-1}{d}  {\revPMP  + o(1)}  }$,  {  {\revPMP knowing that $ R^{ 1}$} is the optimal possible rate in $ \limsup$ by }   \cite{Beck87}.
We have the same result for Coulomb gases but at all scales in $ \mathbb{R}^{ 2}$ thanks to the results of Bauerschmidt et al.  \cite{BauerschmiditBourgadeNikulaYau} or Serfaty  \cite{Serfaty-HU} about smooth linear statistics. The latter work also applies in dimension $ d = 3$, with the exponent $ \alpha  = 1,$ giving a  number variance bound  in $ R^{2 }\ln(R)$, corresponding to the optimal magnitude $ \left(
\mathbf E \left[
\X_{ N}(B(x_{ 0},R))
\right]
\right)
^{ \frac{ 3-1}{3}}$ multiplied by the $ \ln(R)$ term.  Removing this extraneous $ \ln(R)$ term with this strategy would require  improving Serfaty's result to  $ \alpha  > 1 .$

  {\begin{remark}\label{rk:local-averag}
We emphasize that the local averaging 
\(\overline{\mathrm{Var}}^{\varepsilon}\) should not be interpreted as a
smoothing procedure. Rather, it is an artifact of the method, compensating for the lack of strict translation invariance. Indeed, if this averaging was
equivalent to working with a smoother kernel, then one would obtain an
improved optimal variance rate; this is not the case. The best way to see it is with a translation invariant model on the torus. The translation invariance provides also an averaging, but it does not improve the variance rate.

A rigorous proof of this assertion is difficult as it would require, first, establishing a
lower bound in \(V_{\alpha,\gamma}(R)\), for some \(\gamma<\alpha\), in a
non-translation-invariant model with an irregular kernel, at least less
regular than a ball indicator.   One would then have to prove that passing to the 
translation-invariant counterpart model does not decrease the order of magnitude of the
variance.

Let us nevertheless give a concrete illustration in the case of the Gaussian
unitary ensemble (GUE) in dimension \(d=1\). Let
\[
    \varphi(x)=\chi(x)\|x\|^{-a},
    \qquad a\in(0,1/2),
    \qquad \chi\in C_c^\infty .
\]
It can be shown that the supremum of the regularity indices \(\gamma\) such
that \(\varphi\in \mathcal F_\gamma^{ *}\) is
$
    \gamma_a = 1-2a ;
 $
see Section~\ref{sec:periodic}, in particular the proof of
Theorem~\ref{thm:main-periodic}. Using the exact determinantal expression
for the correlation functions, one finds that the variance of the
corresponding linear statistic is of order
 $
    N^{2a}=N^{d-\gamma_a}.
 $

The translation-invariant counterparts of the GUE are the infinite
\(\mathrm{Sine}_2\)-DPP on the real line, or   the periodic
one-dimensional \(\beta\)-ensemble on the torus, also called the log gas, with
\(\beta=2\). These processes are known to be
\(1\)-hyperuniform, with logarithmic variance growth for ball indicators.
Therefore, by Theorem~\ref{thm:main-periodic}, or by
\cite[Proposition~2.2]{survey-hu}, the variance of the linear statistic
associated with the above test function \(\varphi\) is again of order
\(N^{2a}\).

Thus, the variance averaging made possible by translation invariance does
not provide any additional regularity at the level of linear-statistic
variance. The same type of argument should apply to many other
determinantal point processes.
\end{remark}}

  \begin{remark}[Possible extensions] ~\begin{itemize}
\item The proof requires a general preliminary crude bound $  \textrm{Var}\left(\X_{ N}(g _{ R})\right)  = O(R^{ 2d})$ valid for discontinuous  $ g\in \B_{ c}^{}(\mathbb{R}^{ d})$ to deal with some boundary effects, therefore much larger than number variance over ball indicators, and probably over-estimated in many cases. This step is the source of the reduced scales range, and it is not present in the ideal periodic case, where the scales are not truncated, see Section  \ref{sec:periodic}. Improving this bound requires a model dependent strategy, and would allow to weaken the constraint $\maxscale (d + 2)<\frac{ 1}{d}$ and extend the available scales.
\item It could also be possible to choose $ \gamma >1$, but  it leads to further reducing the scales. More generally, the available scales would depend on $ \gamma $, the current setup is adapted to ball indicators and less regular kernels with $ \gamma  \in (0,1]$.
\item The proof applies without modification to general non-negative random Radon measures under the same assumptions, even if they are not atomic.
\end{itemize}  
  \end{remark}

  \subsection{Asymptotic result}  
  \label{sec:asymp}

Let $ \N $ be the  {\it configurations space}, i.e. the class of locally finite sets of $ \mathbb{R}^{ d}$.
Many sequences of particle systems $ \{\X_{ N},N\geqslant 1\}$ are expected to be tight, meaning that   some subsequence    converges  in law to some point process $ \P$ in $ \N$, called  {\it a limit point}. A common framework is to work with  the vague topology $ \xrightarrow[  ]{v}$ generated by  linear statistics with kernel from $ \C_{ c}^{ 0}(\mathbb{R}^{ d})$. The  characterization of weak convergence in law  is the following: we have  $$  \X  _{ N}\xrightarrow[ N\to \infty ]{v} \P \text{\rm{ in law if }}  \X  _{ N}(f)\to \P (f)\text{\rm{ in law for each }} f\in \C_{ c}^{ 0}(\mathbb{R}^{ d}),$$  see for instance  \cite{kallenberg2002foundations}.
If the average number of points in any given fixed ball   is uniformly bounded over $ N$, the sequence of laws $ \{\X_{ N} ;N\geqslant 1\}$ is tight by  \cite[Lemma 14.15]{kallenberg2002foundations}, but this requires working in the space of counting measures, which means points can have a multiplicity higher than $ 1$. We do not detail further here as for the applications we have in mind, a point with a multiplicity larger than $ 1$ is discarded because it would imply an infinite energy. We therefore assume the existence of  a limit point  $\P = \lim_{N_{ j}}{  \X _{ N_{ j}}}$  for some subsequence $ N_{ j}\to \infty $.

So as to delocalize the centering, we choose a deterministic  point $ x_{ N}\in \Sigma $, and consider the rescaled recentered process $\P_{ N} =  \tau _{- x_{ N}}\X_{ N}$, in this way the behavior of $ \X_{ N}$ at any location of $ \Sigma $ can be influential.
For long-range interactions models, the non-uniqueness of the limit $ \P $ is often expected, but in general very hard to prove. It is also expected that ``most'' limit points should be invariant under $ \mathbb{R}^{d}$-translations, but  by no means  automatic.   We must assume that $ \varphi $ is   continuous $ \Leb$-a.e., this is not a strong requirement for a function of $ \F_{ \gamma },$ where $ \gamma >0$.

\begin{theorem}
\label{thm:main-asymp}Let  $ x_{ N}\in  \Sigma $ such that  $ d(x_{ N},\partial  \Sigma ) \to \infty $,    and a limit point $ \P = \lim_{ N_{ j}\to \infty } \tau _{- x_{ N_{ j}}} \X _{ N_{ j}}$.   

 Let $ \alpha \geqslant 0.$ Let $ R_{ N}\to \infty .$ Assume that for some  $ f\in \B_{ c}^{ + }(B_{ 1})$ there is $ C_{ f}$ such that  for all $ N\geqslant 1,R\in [1,R_{ N}]  $,
\begin{align*}
\sup_{ x\in \Sigma ^{ R}} \mathbf E \left[\;\overline{\X_{ N}}(\tau _{ x}f_{ R })^{ 2}
\right]\leqslant C_{ f}R^{ d-\alpha }.
\end{align*}
Then for $\gamma >0,\varphi \in \F_{ \gamma }$, $\Leb$-a.e. continuous,   $R\geqslant R_{ f}^{\circ }, L \geqslant  { R^{ d + 2 + \min(\alpha ,\gamma )}},$  
\begin{align}
\label{eq:concl-asymp-main}
\sup_{ x_{ 0}\in \mathbb{R}^{ d}}\displaystyle\int_{ [-L/2,L/2 ]^{ d}}  \textrm{Var}\left(\P(\tau _{ x_{ 0} + x}\varphi _{ R})\right)\frac{ dx}{L^{ d}}\leqslant C_{ \varphi,f }V_{ \alpha ,\gamma }(R)
 \end{align}
 where $ R^\circ _{ f},C_{ \varphi ,f}$ depend explicitly   on $ f,\varphi ,d,C_{ f}$. 
 
If $ \P$ is furthermore assumed to be stationary,  we have the result for all $   x\in \mathbb{R}^{ d}, R\geqslant R_{ f}^{\circ }$ 
\begin{align}
\label{eq:concl-asymp-main-statio}
  \textrm{Var}\left(\P(\tau _{ x}\varphi _{ R})\right)\leqslant C_{ \varphi,f }V_{ \alpha ,\gamma }(R),
\end{align}
which means that $ \P$ is $ \alpha $-hyperuniform.
 \end{theorem}

 The proof is at Section  \ref{sec:prf-asymp}. We apply this result here to Coulomb systems, a.k.a. one component plasmas. In dimension $ d = 2$, for any $ \beta  >0$,  \cite{Leble-stationary} proved that any limit point $ \P$ is stationary. Hence, the previous result allows us to derive at Section  \ref{sec:coul-infinite} that $ \P$   is $ 2$-hyperuniform, which implies other properties such as finite regularised Coulomb energy or finite Wasserstein distance to the Lebesgue measure, which were not available yet, or number rigidity, already proved in \cite{Leble-stationary}; see Theorem  \ref{thm:spectral-2D-Coul}. For $ d = 3$, it gives $ 1$-hyperuniformity for stationary limit points, the first hyperuniformity result in dimension $ 3$ or more for such models; see also the related work of Chatterjee \cite{ChaHU} for the hierarchical Coulomb gas.

 \subsection{Periodic systems and kernel singularities}
 \label{sec:periodic}
 
 The fundamental result of this article, from which are derived previous results, is Theorem \ref{thm:abstract-general}, adapted to finite periodic systems on the torus. In this setting, the restricted scales range and local averaging caveats are removed. Here we give a   more practical version that can deal with   kernels having  some square integrable singularities, and a converse statement. 
 Let $N \geqslant 1,  \lambda  = N^{ 1/d},$  and $\P_{\lambda   }$  a random measure on the torus $ \T $ endowed with the torus topology and metric. Let $ \tau _{ x}^{\lambda  }$ the $ \lambda  $-periodic translation operator on $ \T[\lambda]$. Say that $ \P_{\lambda  }$ is $ \T[\lambda]$-stationary, or invariant under toroidal translations, if for $ x\in \T[\lambda],$ $$\tau _{ x}^{\lambda  }\P_{\lambda  } \stackrel{(d)}{=}  \P_{\lambda  }.$$ 
 
  For a function $ \varphi \in L^{ 1}(\mathbb{R}^{ d})$ defined possibly outside the torus, extend the definition of linear statistics with 
\begin{align}
\label{eq:extend-statslins-torus}
 \P_{\lambda  }(\varphi ) = \sum_{\k\in \mathbb{Z} ^{ d}}\displaystyle\int_{\T[\lambda]}\varphi (\lambda  \k + x)\P_{\lambda  }(dx).
\end{align}
Note that this is well defined a.s. when $ \P_{\lambda  }$ has Lebesgue intensity, because 
\begin{align*}
 \mathbf E \left[
\P_{\lambda  }( | \varphi  | )
\right] = \sum_{\k}\displaystyle\int_{\T[\lambda]} | \varphi (\lambda  \k + x) | dx = \displaystyle\int_{} | \varphi  | <\infty .
\end{align*}

  Let $a\in (0,d/2), \gamma  = d-2a$. In the proof of the theorem below, we establish that such functions satisfy as in the definition of $ \F_{ \gamma }$
\begin{align*}
  | \hat \varphi (\xi ) |  = O((1 + \|\xi \|)^{ -\frac{ d + \gamma }{2}}).
\end{align*}
Call $ \F_{ \gamma }^{*}$ the vector space generated by  $ \F_{ \gamma }(\mathbb{R}^{ d})$ and functions of the form $ \varphi (x) = \|x-x_{0}\|^{ -a }h (x)$ where $x_{0}\in \mathbb{T}_{1}^{d}, h\in \C_{ c}^{ \infty }(  B _{1/2} )$ is non-negative; and $ \F^{ *}_{ \gamma } :=\F_{ \gamma } $ for $ \gamma \geqslant d$.  
 
 \begin{theorem}
 \label{thm:main-periodic}
  Let $ \P_{\lambda  }$ a stationary point process on $ \T[\lambda]$,   
with $ \mathbf E\left[
 \#\P_{\lambda  }
\right] = \lambda ^{ d}$, and let $ \alpha \in \mathbb{R},\maxscale \in [0,1].   $ 

(i) Assume for some $ \lambda  \geqslant 1$ that for some non-negative $ f\in \B_{ c} (\mathbb{R}^{ d})$ with $ \displaystyle\int_{}f\neq 0,$   there is $ R^{ \circ}\geqslant 1,C_{ f}>0$ such that for $ R\in [ R^{ \circ},\lambda^{ \maxscale }]$
\begin{align*}
  \textrm{Var}\left(\P_{\lambda  }(f_{ R})\right)\leqslant C_{ f}R^{d -\alpha }.
\end{align*}
Then for $ \gamma >0,\varphi \in \F_{ \gamma }^{ *},$  there is $ C_{ \varphi ,f} >0$ such that for $R\in [ R^{ \circ},\lambda^{ \maxscale }]$
\begin{align*}
  \textrm{Var}\left(\P_{\lambda  }(\varphi _{ R})\right)\leqslant
  C_{ \varphi,f}V_{ \alpha ,\gamma }(R).
\end{align*}
 
 A value for $ C_{ \varphi,f }$, likely not far from optimal, can be retrieved from the  constants of  Theorem \ref{thm:abstract-general}.

{ \rev (ii) {\bf Converse direction}. Assume $ \alpha \in (0,2).$  Let  $ f\in \C_{ c}^{ k}(  \mathbb  T  _{ 1}^{ d})$  for some $k>(d + 1)/2 $.
Assume that   there is $R _{ \lambda}\leqslant \lambda^{ \maxscale }$ with $ R_{ \lambda}\to \infty $ such that,  
\begin{align}
\label{eq:limsup}
 \frac{   {\textrm{Var}} \left(\P_{ \lambda}(f_{ R_{ \lambda}})\right)}{V_{ \alpha ,1}(R_{ \lambda})} \to \infty .
\end{align}
Then there is $ R_{ \lambda}'\in [1,R_{ \lambda}]$ such that
\begin{align}
\label{eq:liminf-main}  \frac{  \textrm{Var}\left(\P_{ \lambda}(B(0,R_{ \lambda}'))\right)}{ V_{ \alpha  ,1 }(R'_{ \lambda})} \to \infty .
\end{align}} 
 
  \end{theorem}

 \section{Application to Girko random matrices}
 \label{sec:girko}

 We present   first   a direct application of Theorem \ref{thm:non-asymp} to random matrices. Consider the unordered eigenvalues  $   \X_{ N}: = \{
  {x_{ 1},\dots ,x_{ N}}\}$ of a $ N\times N$ random matrix with  i.i.d. complex entries $ Y_{ i,j}\in  \mathbb C,1\leqslant i,j\leqslant N$ with a fixed continuous law satisfying the following assumption:

  \begin{assumption}\label{ass:girko}  For all $ p>0,$ $ \mathbf E\left[
  | Y_{ i,j} | ^{ p}
\right]<\infty $ and  
\begin{align*}
 \mathbf E [Y_{ i,j}] = \mathbf E [Y_{ i,j}^{ 2}] =  0, \mathbf E  \left[
| Y_{ i,j} | ^{ 2} 
\right]= 1.
\end{align*}
  \end{assumption}
  The traditional scaling is to choose $ \mathbf E \left[
 | Y_{ i,j} | ^{ 2}
\right] =  \frac{ 1}{ N}   $, so that the eigenvalues are concentrated around $ B(0,1)$, but the current microscopic scaling is more adapted to the rest of this article and notationally lighter. 

It is classical that the corresponding renormalized empirical measure  weakly converges to the  {\it circular law}   a.s., i.e. for $ f\in \C_{ c}^{ 0}(\mathbb{R}^{ d})$, a.s.,
\begin{align*}
\frac{ 1}{N} \sum_{i = 1}^{ N}f(N^{ -1/2}x_{ i})\to \frac{ 1}{\pi }\int_{ B(0,1)}f(x)dx,
\end{align*}
see \cite{CipoErdosKol} for more background and references.
It means that the equilibrium measure is indeed proportional to Lebesgue, so that fluctuations of linear statistics have to be assessed  with respect to $ \frac{ 1}{\pi }\Leb$, i.e. we define 
\begin{align*}
 \overline{\X_{ N}}(f) = \X_{ N}(f)-\frac{ 1}{\pi }\Leb(f).
\end{align*}

Cipolloni, Erd\"os and Schr\"oder \cite{CES} give a uniform variance bound for smooth linear statistics  at various scales, see  \cite[(4.4)]{CES}: for $ f\in H_{ 0}^{ 2}(\mathbb{R}^{ d}),  {\revPMP \varepsilon >0}$, there is $ C_{ f,\varepsilon }$ finite such that: for $ \maxscale \in (\varepsilon ,1/2),$ $ R\in [  {\revPMP N^{ \varepsilon }},N^{ \maxscale }]$,
\begin{align*}
  \textrm{Var}\left(\X_{ N}(f_{ R})\right)\leqslant C_{ f  {\revPMP ,\varepsilon }}
\end{align*}
(see  \eqref{eq:bd-CES} in the proof below for details)
which in dimension $ 2$ indicates a $ 2$-hyperuniform behavior.
Note that we only need that this holds with $ \C_{ c}^{ \infty }$ functions here. We can therefore apply the variance transfer principle of Theorem \ref{thm:non-asymp}.

\begin{theorem}Let $ \rho _{ -}<1$  {\revPMP and $\varepsilon >0$ sufficiently small.}
 Let $ \maxscale ,\delta >0$ such that $ 4\maxscale + \delta <  \frac{ 1}{ 2}    $.   Then   for $N\geqslant 1,\varepsilon _{ N}: = N^{ -\delta }$,
     for $\gamma \in (0,1], \varphi \in \F_{ \gamma },  $ there is $ C_{ \varphi   {\revPMP ,\varepsilon } }$ such that for $ R\in [  {\revPMP N^{ q\sqrt{\varepsilon}  }},N^{ b}],$
\begin{align*}
\sup_{ z_{ 0} \in B(0,\rho  _{ -}N^{ 1/d})}\overline{  \textrm{Var} }^{ \varepsilon _{ N}}\left[
{ \X_{ N}}(\tau_{ z_{ 0}}\varphi _{ R} ) 
\right]\leqslant C_{ \varphi  {\revPMP ,\varepsilon }}R^{ 2-\gamma  {\revPMP  +\sqrt{ \varepsilon} }},
\end{align*} 
recalling $ q = 4 + 1\vee (\gamma -2)$. \end{theorem}  
For $ \gamma  = 1,\varphi  = 1_{ B(0,1)}, $  we  {\revPMP almost}  recover the  surface order bound for  balls number variances $ V_{ 2,1} (R) = R^{ 1  }$, with $\X_{ N}(\tau _{ z_{ 0}}\varphi _{ R}) = \X_{ N}(B(z_{ 0},R))$.

  This result therefore formally prevents the number variance in balls to have a power law growth $ \geqslant R^{ a},a>1,$ outside a   set of centers $ x$ with vanishing density. See Remark~\ref{rk:local-averag} for a discussion explaining why the local averaging $ \overline{\mathrm{Var}}^{\varepsilon_N} $ should not be interpreted as a smoothing procedure.  {\revPMP The rate does not hold until the microscopic scale $ R\sim 1$ because the variance bound in  \cite{CES} does not hold uniformly until the microscopic scale (the constant involved does depend on $ a$ when $ a\to 0$).}

   {Since the precise statement we need is not stated as a stand-alone result, let us indicate how to retrieve it. Let $ f\in \C_{ c}^{ \infty }(B_{ 1})$ non-negative and non-degenerate. 
Let us introduce some notation from  \cite{CES}. Let $ a\in [0,1/2)$ and  $ R_a = N^{ 1/2-a}$. They define at (2.2)
\begin{align*}
 L_{N}(f_{ z_{ 0},a}) := \X_{ N}(\tau _{ z_{ 0}}f_{ R_a})-\mathbf E \left[
\X_{ N}(\tau _{ z_{ 0}}f_{ R_a})
\right].
\end{align*}
They give in (2.5) an expectation estimate, which yields
\begin{align*}
  | \mathbf E\left[
 \X_{ N}(f_{ R_a})
\right]-\frac{ 1}{\pi }\Leb(f_{ R_a}) | \leqslant C_{ f}.
\end{align*}
Let $ \tau  =  \frac{ 1}{ 2}   (1-\rho _{ -})$.
They give in (4.4) the variance estimate  (with $ f^{ (i)} = f^{ (j)} =   f$), if $  z_{ 0}\in B(0, (1-\tau )N^{ 1/d}),$
\begin{align}
\label{eq:bd-CES} 
 \mathbf E \left[
L_{N }(f_{ z_{ 0},a})^{ 2}
\right] 
 =\frac{ 1}{4\pi } \int_{ } \nabla f^{ 2} + C_{ f}(  N^{ -4a} + N^{ -c}).
\end{align}
All in all it yields for us 
\begin{align*}
\mathbf E\left[\;
 \overline{\X_{ N}}(\tau _{ z_{ 0}}f_{ R_{  { a}}})^{ 2}\right]\leqslant 2\mathbf E \left[
L_{N }(f_{ z_{ 0},a})^{ 2} 
\right] + 2| \mathbf E \X_{ N}(f_{ R_a})-\frac{ 1}{\pi }\Leb(f_{ R_a}) |^{ 2} \leqslant C_{ f}.
\end{align*}
This inequality  only holds uniformly for $ a\in (0,1/2  {\revPMP -\varepsilon } ]$  {\revPMP  for any $ \varepsilon>0$}.
 We only need to treat the exponents  $ a \geqslant 1/2-\maxscale $. Since $ \{R_{ a};a\in [1/2-\maxscale,1/2  {\revPMP -\varepsilon }  ]\} = [  {\revPMP N^{ \varepsilon  }},N^{ \maxscale }]$, we indeed have for all $ R\in [  {\revPMP N^{ \varepsilon  }},N^{ \maxscale }]$
\begin{align*}
 \sup_{ z_{ 0}\in B(0,\rho _{ -}N^{ 1/d})}\mathbf E \left[\;\overline{\X_{ N}}(\tau _{ z_{ 0}}f_{ R})^{ 2}
\right]\leqslant C_{ f}.
\end{align*}

We can therefore apply  {\revPMP point (ii)} of Theorem \ref{thm:non-asymp} with $d =  \alpha  = 2$, concluding the proof.\\

 {\bf Comparison with  Cipolloni, Erd\"os and Kolupaiev  \cite{CipoErdosKol}}. We compare with Theorem 2.4   from the recent work  \cite{CipoErdosKol}, where  hyperuniformity of Girko eigenvalues was proven for the first time. The two most important points are \begin{itemize}
\item they  have a non-optimal variance rate $ R^{ 2-\varepsilon }$ with $ \varepsilon \leqslant \frac{ 1}{20}$ instead of the   {\revPMP almost}  optimal rate $ R^{  { 1  {\revPMP  + o(1)}}}$ obtained here,
\item On the other hand their results are valid up to scale  $ R\sim N^{ 1/2},$ and without the local averaging $ \overline{  \textrm{Var}}^{ \varepsilon _{ N}}$  {\revPMP (note that the microscopic scale  {\revPMP $a =  \frac{ 1}{ 2}   $} is not available either in their main results)}
\end{itemize}  Also :\begin{itemize}
\item  {\bf Kernel.} They consider indicator functions of the form $ \varphi  = 1_{\Omega  }$ where $ \Omega \subset B(0,\rho _{ -})$ is simply connected with $ \mathscr{C}^{ 2}$ boundary and has uniformly bounded curvature. In the current article, we consider not only indicator functions, but any function $ \varphi \in \B_{ c}^{}(\mathbb{R}^{ d})$, as long as it is  in $ \F_{  }: = \cup _{ \gamma >0}\F_{ \gamma }$ (but the variance rate would be in $ V_{ 2 ,\gamma }(R)$). For $ \varphi  = 1_{ \Omega }$, $ \varphi $ is in $ \F_{ }$ if $ \Omega $ is compact and convex and its boundary is $ 5/2$-times differentiable with non-vanishing curvature by  \cite[Thm. 2.16]{IoseElf}.
\item  {\bf Real case.} They obtain a similar result  for the real case but with smaller $ \varepsilon \leqslant  \frac{ 1}{ 53}.   $ In our case, we rely on Remark 3.6 from  \cite{CES} stating that the real case could be obtained with similar methods; and this could in principle be transferred with Theorem \ref{thm:non-asymp} to irregular kernels.
\item  {\bf Entries density.} In  \cite{CipoErdosKol}, in addition to Assumption  \ref{ass:girko}, they require a density in $ L^{ 1 + a}(  \mathbb C )$ for the $ Y_{ i,j}$, for some $ a>0$, but this assumption seems to be removable by their Remark 2.2. 
\end{itemize}

 \section{Applications to Gibbs measures}  
 \label{sec:gibbs}

We apply here the transfer principle to existing bounds on smooth  linear statistics over Gibbs measures, which requires a bit more preparation.  Let $ N\geqslant 1$
 and $ \X_{ N}^{ \beta }\subset \mathbb{R}^{ d}$  a random $ N$-tuple of points with density, in $X = (x_{ 1},\dots ,x_{ N}), $
\begin{align}
\label{eq:boltz}
\frac{ 1}{Z_{N }^{ \beta }}\exp(-\beta (H(X) + C(X))),
\end{align}
  where $ H$ is a measurable ``energy'' function, $ \beta  = \beta _{ N}>0$ is  the inverse temperature,  $ C(X)$ is a confinement term, and $ Z_{ N  }^{ \beta }$, the  {\it partition function}, is a renormalizing constant. We consider here pairwise interaction energies 
\begin{align}
\label{eq:hamil}
 H(X) = \sum_{i< j}\psi (\| x_{ i}-x_{ j}\|)
\end{align}
for some $ \psi : \mathbb{R}^{ d}\times \mathbb{R}^{ d}\to \mathbb{R}.$  See the recent monographs  \cite{Lewin,serfaty-book,bookDereudre} for an exhaustive introduction to the mathematical treatment of such models.
 The corresponding linear statistics are denoted by 
\begin{align*}
 \X_{ N} ^{ \beta }(f) = \sum_{i = 1}^{ N}f(x_{ i})
\end{align*} for $ f\in \B_{ c}^{}(\mathbb{R}^{ d})$. Even though we introduced $ \X_{ N}^{ \beta }$ as a tuple of points, we rather consider by an abuse of notation that $ \X_{ N}^{ \beta }$ is an unordered set of $ N$ points, i.e. a point process, as in the rest of the article, and it bears no consequences as we only consider linear statistics, which are functionals symmetric in the $ N$ points.

\subsection{Number variance for 2D and 3D Coulomb gases}

The Coulomb gases  are defined in dimension $ d\geqslant 2$ by  \eqref{eq:boltz}-\eqref{eq:hamil} with the Coulomb potential 
\begin{align*}
 \psi_{ d} (t) = 
 \begin{cases} 
 -\ln(t)$  if $d = 2\\
t^{ 2-d}$  if $d\geqslant 3.
  \end{cases}
\end{align*}
  This potential is   relevant for many models  in physics \cite{Lewin,serfaty-book}, mathematically  convenient due to the homogeneity of $ \psi _{ d}$ and its super-harmonicity,  $ \Delta \psi _{ d}(\|\cdot \|) =- c_{ d}\delta _{ 0}$ with $ c_{ 2} = 2\pi , c_{ d} = (d-2) |  \mathbb S^{ d-1} | ,d\geqslant 3.$ The corresponding  point process is denoted here $ \X_{ N}^{ d,\beta }$ for $ \beta >0.$ We consider the canonical choice for the confinement term  \begin{align*}
 C(x_{ 1},\dots ,x_{ N}) = \frac{ c_{ d}}{2d} \sum_{i = 1}^{ N} \|x_{ i}\|^{ 2}.
\end{align*}

The system $ \X_{N }^{ d,\beta }$ is not formally invariant under translations, but it is expected to be quasi homogeneous  in the bulk. Some works  rather consider Coulomb gases without confining potential and with periodic boundary conditions, i.e. $ C(X) =0$ with $ \lambda  = N^{ 1/d}$ and 
\begin{align*}
H(X) = \sum_{1\leqslant i< j\leqslant N}\psi_{ d}^{ \T} (X)
\end{align*}
where $ \psi _{ d}^{ \T}$ satisfies by analogy $ \Delta \psi _{ d}^{ \T} =c_{ d}(\lambda ^{ -d}\Leb-\delta _{ 0})$ on $ \T$  \cite{CohHer}.
 The periodic transfer principle of this article is better suited to processes which are  invariant, as it applies with unrestricted scales, so ideally it would apply at all scales to the periodic   version of the model on $ \T$. Unfortunately, to the author's knowledge, no  hyperuniform bounds are available for the variance of smooth linear statistics, the work \cite{CohHer}  for instance only treats the macroscopic scale whereas we need bounds down from small  scales to apply Theorem \ref{thm:abstract-general}. To have an optimal transfer principle in the non-asymptotic setting, a promising direction is to reproduce the results of Serfaty  \cite{Serfaty-HU} in a periodic setting.\\
 
 It is known that most points will typically spread in the  {volume $ N$} equilibrium disk $ B_{\rho ^{+}N^{ 1/d}}$ (i.e. $ \rho _{  + } = \Leb(B_{ 1})^{ -1/d}$), also called the  {\it droplet},   according to the   macroscopic uniform equilibrium measure $$ \mu_{ N}  = { {\bf 1}\{ B_{ \rho ^{+}{N^{ 1/d}}} \}}\Leb $$  \cite[Section 2.1]{serfaty-book}. The set of points of $ \X_{N}^{d,\beta }$ close to the boundary, or even outlier points outside the disk, have a different behavior  not captured by general statements. They are distinguished from the ``bulk'', consisting of interior points far from the boundary. We fix below  $ c>0$ sufficiently large and  
\begin{align}
 \label{eq:rho-plus}
\rho ^{  - }_{ N} = \rho  _{  + }N^{ \frac{ 1}{d}}-cN^{ \frac{ 1}{d + 2}}  .
\end{align}
and we only treat  the points of $ \X_{N}^{d,\beta }\cap B_{\rho ^{-}_{N}}$.  \\

A longstanding questions is whether Coulomb systems are  hyperuniform, i.e. whether the variance of the number of points in large balls have reduced fluctuations, see  \cite{Lewin,Leble,ChaHU} and references therein.  
 The only  hyperuniformity result in dimension $ d\geqslant 2$ for invariant Coulomb systems is Leblé \cite{Leble},  with a logarithmic reduction on the Poisson  rate for the number variance over a ball $ B(x,R)$, i.e. the variance is bounded by $ CR^{ 2}\ln(R)^{ -0.6}$. Chatterjee  \cite{ChaHU} gives  sharp hyperuniformity  results for a simpler model, the  {\it hierarchical Coulomb gas}, that is not stationary and characterized by conditional independence of subsystems and a self-similar nature.

One of the main points of the current article is to state that  hyperuniformity is not about the number variance over balls, but about the variance decay for any rescaled linear statistic, and that it should transfer to balls up to the exponent truncation at the regularity index $ \gamma $, with $ \gamma  = 1$ for the ball. The reduced scale range for the current results $ [1,N^{\maxscale }]$ instead of $ [1,N^{ 1/d}]$ is  an artifact of the method.

  For smooth linear statistics, Serfaty  \cite{Serfaty-HU} derived reduced variance bounds for large Coulomb systems, with variance rate $V(R)\asymp R^{ 2d-4}$ (in the current scaling), i.e. $ R^{ 0} = R^{ 2-2}$ in dimension $2$, also called $ 2$-hyperuniformity (negative exponent $ \alpha  = 2$) and $ V(R) \asymp R^{ 2} = R^{ 3-1} $ in dimension $ 3$, or $ 1$-hyperuniformity ($ \alpha  = 1$).  Bauerschmidt et al.  \cite{BauerschmiditBourgadeNikulaYau} give similar results  in dimension $ 2$. Serfaty \cite{Serfaty-HU} also gives a CLT with renormalisation in $ \sqrt{R^{ d-2}}$, suggesting that the optimal exponent is $ 2$, as for most known disordered systems  in dimension $ d\geqslant 2$ \cite{survey-hu}, in particular the exponent $ 2$ in dimension $ 2$ is likely optimal. In view of Theorems \ref{thm:non-asymp}, \ref{thm:main-asymp} we seek to achieve variance rates of order $ V_{\alpha ,1}(R)$ over balls with radius $ R$, i.e. with the rate  $$ V_{ d}^{ \Coul}(R) := 
\begin{cases} 
R$  if $d = 2\\
R^{ 2}\ln(R)$  if $d = 3.
 \end{cases}$$ 
 
 \subsubsection{Infinite volume Coulomb gases}  
 \label{sec:coul-infinite}

 As in Section \ref{sec:asymp}, we give first a result for a  limit point, 
 i.e. a point process    limit of a subsequence $$\P^{ d,\beta } =  \lim_{ N_{ j}\to \infty }\tau _{ -x_{ N_{ j}}}\X_{ N_{ j}}$$ weakly in the vague topology, with $ x_{ N}\in B(0,\rho _{ N}^-)$ deterministic. A deeper interpretation of  the transfer principle in the stationary case involves the second order spectral measure, and this in turn implies global properties related to Coulomb energy, optimal transport, rigidity, and decorrelation exponent, that we state just after.
 
 The existence of such limit points in dimension $ d\geqslant 2$, that is, the tightness of the laws of the $ \X_{ N},N\geqslant 1$, has been proved by Armstrong and Serfaty  \cite{ArmSerf}. This tightness is more generally a consequence of a uniformly bounded intensity by \cite[Thm. 14.16-(iv)]{kallenberg2002foundations}, hence it can immediately be deduced from the estimates of \cite{Serfaty-HU}, see in particular Equation  \eqref{eq:pre-fluct-exp} below. The fact that the limit is indeed a simple point process and not a counting measure with multiplicity higher than $ 1$ is a consequence of  \cite[(1.21)]{ArmSerf}. The unicity of the limit $ \P$ is in general not expected to hold, especially at low temperature, but this is a mathematically  difficult question. Whether $ \P$ is stationary is also not automatic; Leblé  \cite{Leble-stationary} proves that any such process in dimension $ 2$ is translation invariant, which we integrate to the result below.
 As in the rest of the article, we also give results for  linear statistics less regular than the ball indicator.
 
\begin{theorem}[Asymptotic  hyperuniformity for $ d = 2,3$]
\label{thm:3d-coul}
Let $ x_{ N}\in B(0,\rho _{ N}^-),N\geqslant 1$, $ \P^{ d,\beta }$ a limit point of $   \{  \tau _{ -x_{ N}}\X^{ d,\beta }_{ N};N\geqslant 1\}$ in the vague topology, $ \gamma>0,\varphi \in \F_{ \gamma }$ continuous  $ \Leb$-a.e..
\label{thm:Coul-infinite} 
\begin{itemize}
\item  {\bf 2D:}  $  \P^{ 2,\beta }$ is stationary (by  \cite{Leble-stationary}) and $ 2$-hyperuniform:  as $ R\to \infty $ 
\begin{align*}
  \textrm{Var}\left(\P^{ 2,\beta }( \varphi _{ R})\right)  = O(R^{ 2-\min(\gamma,2) }\ln(R)^{ {\bf 1}\{ \gamma  = 2 \}}).
 \end{align*}
\item  {\bf 3D:} There is $ R^{ \circ}\geqslant 1$ such that for $R\geqslant R ^{\circ }, L \geqslant  { R^{6}},$  
\begin{align}
\label{eq:concl-asymp-main}
\sup_{ x_{ 0}\in \mathbb{R}^{ d}}\displaystyle\int_{ [-L/2,L/2 ]^{ d}}  \textrm{Var}\left(\P^{ 3,\beta }(\tau _{ x_{ 0} + x}\varphi _{ R})\right)\frac{ dx}{L^{ d}}\leqslant C_{ \varphi }R^{ 3-\min(\gamma,1) }\ln(R)^{ {\bf 1}\{ \gamma  = 1 \}}.
 \end{align}
 If $ \P^{ 3,\beta }$ is stationary, it is $ 1$-hyperuniform: 
\begin{align*}
  \textrm{Var}\left(\P^{ 3,\beta }(\varphi _{ R})\right) = O(R^{ 3-\min(\gamma ,1) }\ln(R)^{ {\bf 1}\{ \gamma  = 1 \}}).
\end{align*}
\end{itemize}
 \end{theorem}
 
  {
 We also have the possibility to restrain the possible  {constant sign }decorrelation exponents. The following proposition yields that, if the covariance admits a decay with exponent $ a$ on a sufficiently large subset of $ \mathbb{R}^{ d}$,  necessarily $ a\geqslant 4$.
To state it formally,  let us introduce the reduced covariance measure  $   \mathsf C$. This measure is characterized over real Schwartz functions $ f,g$ by
\begin{align*}
 \textrm{Cov}\left(\X(f),\X(g)\right)=\int_{\mathbb{R}^{d}} f(x)g(x+y)dx   \mathsf C(dy),
\end{align*} 
see  \cite{DVj08} for proof of existence.
For an interpretation of $   \mathsf C(dx)$ for $ x\neq 0$, take as $ f,g$ approximations of Dirac masses around resp. $ 0$ and $ x$, $   \mathsf C(dx)$ denotes the covariance  of infinitesimal masses at $ 0$ and $ x.$ $   \mathsf C$ usually has a Dirac mass at $ 0$, and for disordered processes, it is often expected or assumed that $   \mathsf C$ has a density $   \mathsf c(x)$ for $ x\neq 0$, and the decay of $   \mathsf c(x)$ as $ x\to \infty $ reflects the asymptotic decorrelation of the process. 

 Call truncated cone a set of the form $ \Gamma  = C  \setminus B(0,r)$ where $ C$ is a non-empty open cone and $ r\geqslant 0.$

 \begin{proposition}
 \label{prop:decorr}
 Let $ \P$ a stationary Coulomb gas in dimension $2$  or $3$ as in Theorem  \ref{thm:3d-coul}. If for some truncated cone $ \Gamma $, the covariance measure $   \mathsf C$ of $ \P$ satisfies 
\begin{align*}
    \mathsf C(dx) \geqslant K\|x\|^{ - a}dx, x\in \Gamma ,
\end{align*} for some $ K>0$, or if $ -   \mathsf C$ satisfies it,
then $$ a\geqslant 4 .$$
 \end{proposition}
The technique  could likely extend to prove $ a\geqslant 4$  for the Coulomb gas in any dimension, but it is sharp only in dimension $ 2$, we would expect $ a\geqslant d + 2$ in general.  {There is on the other hand no converse statement, in the sense that even if $  0\leqslant   \mathsf C(dx)\leqslant K\|x\|^{ - a}$ on all $ \mathbb{R}^{ d}$, then there is no bound on $ a$. Ginibre ensemble, the 2D Coulomb gas at inverse temperature $ \beta  = 2$, has a deccorelation with Gaussian tail, meaning $ a$ can be arbitrarily large.}
 
  {The proof requires the concepts of spectral measure, so we  postpone it to subsection  \ref{sec:coulomb-spectral} below.} }

  \newcommand{\W}{   \mathsf W}  

\subsubsection{Consequences on  rigidity, transport, and Coulomb energy}
\label{sec:coulomb-spectral}

 {The concept of $ \alpha $-hyperuniformity is more simple for stationary processes over $ \mathbb{R}^{ d}$ as there is no ``scale'' involved, at the contrary of large finite systems. As we saw at Theorems  \ref{thm:main-asymp} and  \ref{thm:3d-coul} above, a stationary point process $ \P$ is $ \alpha $-hyperuniform if for some non-negative $ f\in  \mathscr  B_{ c}  \setminus \{0\},$
\begin{align*}
  \textrm{Var}\left(\P(f_{ R})\right) = O(R^{ d-\alpha }).
\end{align*}}
At the heart of considerations about fluctuations of linear statistics for a locally square integrable stationary point process $ \P$ lies  the concept of spectral measure $ \S$, defined below.   The results above imply $ \alpha$ -hyperuniformity for $\alpha  \geqslant 2$ for $ \P^{2,\beta }$, and this in turn implies several macroscopic properties. Let us  briefly introduce the related concepts. \begin{itemize}
\item The  {\bf spectral measure} $ \S$ of a locally square integrable  stationary point process $ \P$ is characterized by
\begin{align}
\label{eq:PS-formula}
  \textrm{Var}\left(\P(f)\right) =   \int_{ \mathbb{R}^{ d}} | \hat f | ^{ 2}\S, f\in \B_{ c}^{ }(\mathbb{R}^{ d}).
\end{align}
It is the infinite volume continuous space  analogue of the spectral measure of time series, or of periodic processes over the torus,  see Proposition \ref{lm:spectral}. $ \S$ is also the Fourier transform in the distributional sense of the reduced covariance measure  $ \C$ \cite{DVj08,survey-hu}.  { $ \alpha $-hyperuniformity is in fact equivalent to $ \S(B_{ r}) =O(r^{ d + \alpha }) $ as $ r\to 0$ ( \cite{HartBjo}, \cite[Prop. 2.2]{survey-hu})}.
\item One way to assess the large scale order of a point process  $ \P$   is through its distance to Lebesgue measure in terms of  {optimal transport}.   The most commonly used metric is the Wasserstein distance $ \W_{ p},p\geqslant 1$, see  \cite{LrY,HueLebl}.  Recall that between two measures $ \mu ,\nu $ with finite identical mass on $ \mathbb{R}^{ d}$, 
\begin{align*}
  \mathsf W_{ p}^{ p}(\mu ,\nu ) = \inf_{ M}\int_{ \mathbb{R}^{ d}\times \mathbb{R}^{ d}}\|x-y\|^{ p}M(dx,dy)
\end{align*}
where the infimum is over couplings $ M$ of $ \mu $ and $ \nu $, i.e. $ M$ is a measure over $ \mathbb{R}^{ d}\times \mathbb{R}^{ d}$ such that $ M(\mathbb{R}^{ d},\cdot ) = \nu ,M(\cdot ,\mathbb{R}^{ d}) = \mu .$ Extending this concept to measures with infinite mass requires some care,
see for instance
Erbar et al. \cite[(1.2)]{erbar2023optimal}, where  the Wasserstein distance between unit intensity stationary random measures is defined as an average per unit volume.  
Thanks to  \cite[Proposition 2.11]{erbar2023optimal}, a finite such distance per unit volume can be characterised in more concrete terms through the concept of  {\it perturbed lattice}. Say a unit intensity point process $ \P$ is a $ L^{ p}$-perturbed lattice if there exists a $ \mathbb{R}^{ d}$-valued stationary field $ \{X_{ \k};\k\in \mathbb{Z} ^{ d}\}$, i.e. whose law is invariant under $ \mathbb{Z} ^{ d}$ translations, such that  $ \mathbf E \|X_{ 0}\|^{ p}<\infty $ and
\begin{align*}
 \P \equlaw \{\k + X_{ \k};\k\in \mathbb{Z} ^{ d}\}.
\end{align*}
It is proved at  \cite[Prop. 2.11]{erbar2023optimal} that $ \P$ has a finite  $ p$-Wasserstein distance to Lebesgue measure per unit volume  if and only if 
 $ \P$ is a $ L^{ p}$-perturbed lattice.

\item The  {\bf Coulomb energy}, ubiquitous in the study of Coulomb gases  \cite{Lewin,serfaty-book,HueLebl},
can be defined for a finite sample $ \X_{ N}$ with $ N$ points on $   \mathbb  T  _{ N^{ 1/d}} $  in the following way:
\begin{align*}
 H_{ N}(\X_{ N}) = \frac{ 1}{2}\displaystyle\int_{\mathbb  T  _{ N^{ 1/d}} }\displaystyle\int_{\mathbb  T  _{ N^{ 1/d}} }{\bf 1}\{ x\neq y \}\psi _{ d}(\|x-y\|)(d\X_{ N}(x)-dx)(d\X_{ N}(y)-dy
).\end{align*}
Due to the heavy tail of the Coulomb potential, defining an infinite volume analog for a stationary limit point $ \P$ is more challenging, see   the discussion at  \cite[Section 1.3]{HueLebl}.  The general idea to circumvent this issue is to  relate the Coulomb energy to the well-definedness  of the stationary electric field generated by particles placed at points of $ \P$, and the finiteness of its second moment. This question has been tackled by Sodin et al.  \cite{SodinElectric} for a stationary point process $ \P$, showing that a necessary and sufficient condition  is expressed with the spectral measure $ \S,$  namely that  \begin{align}
\label{eq:finite-Coul}
 \int_{B(0,1 )}\|\xi\|^{ -2}\S(d\xi )<\infty .
\end{align}
    Building on  \cite[Theorem 4]{HueLebl} we will say here that a stationary point process $ \P$ has  {\it  finite regularised Coulomb energy} if  its spectral measure satisfies \eqref{eq:finite-Coul}.

It has been noted in  \cite{LrY} that under an additional integrability condition,  finite Coulomb energy has implications in terms of optimal transport; Huesmann and Leblé \cite{HueLebl}   showed a stronger result without integrability assumptions: finite regularised Coulomb energy implies a finite $  \overline{ \mathsf W}_{ 2}^{ 2}(\P,\Leb[2])$ distance.  
  Hence the points of the plane $ \mathbb{R}^{ 2}$  can be allocated to the points of $ \P$ with mass conservation and finite average quadratic cost; equivalently, it means that there is   a random  equivariant  bijection between $ \P$ and a properly rescaled version of $ \mathbb{Z} ^{ 2}$ such that the expected square matching distance is finite.
  
 \item The  {\bf linear number rigidity} of $ \P$ means that the number of points in a compact $ K$ with non-empty interior (typically $ K = B_{ 1}$) can be approximated by linear statistics supported by $ K^{ c}$:
\begin{align}
\label{eq:number-rigid}
\inf_{ h\in \C_{ c}^{ \infty }(K^{ c})}\mathbf E\left[
 \left|
\P(K)-\P(h)
\right|^{ 2}
\right] = 0,
\end{align}
this is in theory stronger than the concept of rigidity investigated in the last years  \cite{GP17,Lr24,ChhaibiNaj,DHLM}, but in many cases the rigidity of such disordered processes is linear.
 The concept emerged after the seminal work of  \cite{GP17} about the infinite Ginibre  ensemble and the zeros of the planar Gaussian analytic function, two emblematic examples in the theory of  hyperuniformity. See  \cite[Chapter 5]{survey-hu},  \cite{Lr24} for the spectral characterization of linear rigidity and further references. We prove below number rigidity for the 2D Coulomb gases, it was already proved in  \cite{Leble-stationary} by different methods based on DLR equations (not involving linear rigidity). We can actually give a convenient general lemma, consequence of  \cite{Lr24}:

\begin{lemma}\label{lm:NR}
 Let $ \X$ be a $ d$-hyperuniform stationary point process on $ \mathbb{R}^{ d}$, i.e. $ \S(B_{ \varepsilon }) = O(\varepsilon ^{ 2d})$.
Then $ \X$ is linearly number rigid.
\end{lemma}

\begin{proof}
The assumption means for $ \varepsilon <\varepsilon _{ 0}$
\begin{align*}
 \S(B_{ \varepsilon }) \leqslant C\varepsilon ^{ 2d}.
\end{align*}

 Denote by $ \s(\xi)d\xi $ the component of $ \S$ continuous  with respect to Lebesgue measure. Clearly 
\begin{align*}
 \int_{ B_{ \varepsilon }}\s(\xi)d\xi \leqslant \S(B_{ \varepsilon })\leqslant C\varepsilon ^{ 2d}.
\end{align*} \cite[Theorem 1]{Lr24} yields number rigidity if for all $ \varepsilon >0$
\begin{align*}
 \int_{ B_{ \varepsilon }}\frac{ 1}{\s(\xi)}d\xi = \infty .
\end{align*}  
Assume the integral is finite for some $ \varepsilon_{ 0} >0$. Then for $ \varepsilon $ sufficiently small,
\begin{align*}
 \int_{ B_{ \varepsilon }}\frac{1 }{\s(\xi)}d\xi <\Leb(B(0,1))^{ 2}/C.
\end{align*} Then Cauchy-Schwarz inequality yields
\begin{align*}
\Leb(B(0,1))\varepsilon ^{ d} = &\int_{ B_{ \varepsilon }}\sqrt{\s(\xi)} \frac{1}{\sqrt{\s(\xi)}}d\xi \leqslant \left(
\int_{ B_{ \varepsilon }}\s(\xi)d\xi 
\right)^{  \frac{ 1}{ 2}   } \left(
\int_{ B_{ \varepsilon }}\frac{ 1}{\s(\xi)}d\xi \right) ^{  \frac{ 1}{ 2}   }   \leqslant \sqrt{C}\varepsilon ^{ d}\left(
\int_{ B_{ \varepsilon }}\frac{ 1}{\s(\xi)}d\xi \right)^{  \frac{ 1}{ 2}   } \\
<&\Leb(B(0,1))\varepsilon ^{ d},
\end{align*}
reaching a contradiction.
\end{proof}
\end{itemize}
The hyperuniformity rates obtained above allow to assess that these properties are satisfied by infinite volume 2D Coulomb gases.
\begin{corollary}
\label{thm:spectral-2D-Coul}
Any  limit point $ \P^{ 2,\beta }$ as in Theorem  \ref{thm:3d-coul}  is stationary and has a spectral measure satisfying $ \S^{ 2,\beta }(B_{ \varepsilon }) = O(\varepsilon ^{ 4})$ as $ \varepsilon \to 0$. Hence we have\begin{itemize}
\item  {\bf Finite regularised Coulomb energy}:  
\begin{align*}
 \int_{ B_{1 }}\|\xi\|^{ -2}\S^{ 2,\beta }(d\xi )<\infty .
\end{align*}
\item  {\bf Finite $ 2$-Wasserstein distance:} $  \overline{ \W}_{ 2}^{ 2}(\P^{ 2,\beta },\Leb)<\infty $.
\item  {\bf Linear number rigidity}, i.e.  \eqref{eq:number-rigid} holds.
\end{itemize}
\end{corollary}

In dimension $ d\geqslant 3$, these questions have not the same relevancy. A standard point process has finite $ \overline{\W}_{ p}^{ p}$-distance to Lebesgue measure for all $ p$ (the homogeneous Poisson process does   if $ d\geqslant 3$  \cite{AKT}), see  \cite{LrY}, and  $ k$-rigidity typically does not occur in dimension $ d\geqslant 3$  \cite{Lr24}. An interesting question, in dimension $ 2,3$ or more, would be to have exponential tails on the typical transport cost, like for the zeros of the planar GAF  \cite{NSV}, or the Poisson process in higher dimensions .

\begin{proof}The consequence of Theorem  \ref{thm:3d-coul}  is $ 2$-hyperuniformity,  i.e. $  \textrm{Var}\left(\P^{2,\beta }(f_{ R})\right)=O(1)$ for sufficiently  smooth $ f$. Hence      \cite[Proposition 2.2]{survey-hu} yields that the spectral measure of $ \P^{2,\beta }$ satisfies $ \S(B_{ \varepsilon }) = O(\varepsilon ^{ 4})$. 
\begin{itemize}
\item This readily  implies finite regularised Coulomb energy.
\item The property $ \overline{\W}_{ 2}^{ 2}(\P^{ 2,\beta },\Leb)<\infty $ is  a consequence of the finite Coulomb energy by  \cite[Thm 1]{HueLebl}.
\item Linear number rigidity follows from Lemma  \ref{lm:NR}.

\end{itemize}
\end{proof}

Linear  number rigidity means that the moment of order $ 0$ of the finite measure $ \P1_{ B(0,1)}$ can be perfectly interpolated from $ \P1_{ B(0,1)^{ c}}$. Some models, in particular the GAF zeros  \cite{GP17}, experience $ k$-rigidity for some $ k\geqslant 1$, meaning that the moments of order $ k$  can also be interpolated.
It was shown in  \cite{Lr24} that in many cases, linear $ k$-rigidity is equivalent to  the non-integrability in $ 0$ of $ \|\xi\|^{ 2k}\s(\xi)^{ -1}$, in particular this equivalence holds for point processes which are rotationally invariant, but formally one needs more information on the structure factor.   For $ d = 2,3,$ it is expected that the exponent $ 2$ is optimal, i.e. the spectral density (if it exists) should satisfy $    \mathsf s(\xi)\sim \|\xi\|^{  2}  $, which is consistent with the bound $ \S(B_{ \varepsilon }) = O(\varepsilon ^{ d + 2})$ derived above in dimension $ 2$. Hence $ \|\xi\|^{ 2}\s(\xi)^{ -1}$ is likely integrable in $ 0$, meaning that $ \P^{ 2,\beta }$ is not $ 1$-rigid. In dimension $ 3$, it would yield that $ \s(\xi)^{ -1}$ is  integrable, and therefore not linearly number rigid.

\subsubsection{  Proof of  Proposition \ref{prop:decorr}}
 
 \label{sec:proof-decorr}
 {
We show more generally that for $ \P$ a stationary $ \alpha $-hyperuniform process in dimension $ d$, if the covariance measure satisfies $   \mathsf C(dx)\geqslant K\|x\|^{ - a}$ on a truncated cone $ \Gamma $, then $ a\geqslant d + \alpha $. The 2D and 3D Coulomb gases are obtained for $ \alpha  = 4-d$.

 Let \(\varphi\in \C_c^\infty \), \(\varphi\not\equiv 0\) non-negative with support contained in $ B(0,\rho /2)$ for some $ \rho >0$, and set

\[
\psi=\varphi * \widetilde\varphi ,
\qquad 
\widetilde\varphi(x)= {\varphi(-x)}.
\]

Then \(\psi\geqslant 0\), and

\[
\widehat\psi(\xi )=|\widehat\varphi(\xi )|^2\geqslant 0.
\]
We have  $  { \rm supp}(    \psi) \subset B_{ \rho }$. Let  $ y$ such that $ B(y,\rho )\subset \Gamma .$ Define for $ R>1$
\[
\psi_{R,y}(x)
=R^{ -d}
 \psi\!\left(\frac{x-Ry}{R}\right),
\]it
does not vanish only if $ x\in B(Ry,\rho R) = RB(y,\rho )\subset R\Gamma \subset \Gamma $, hence $ \psi _{ y,R}$ is supported by $ \Gamma .$

Its Fourier transform is
\[
\widehat{\psi_{R,y}}(\xi )
= 
e^{-iRy\cdot \xi }\widehat\psi(R\xi ).
\]

Therefore, by  \eqref{eq:PS-formula},
\begin{align*}
\langle   \mathsf C,\psi_{R,y}\rangle
=&
\int_{\mathbb R^d} e^{-iRy\cdot \xi }\widehat\psi(R\xi )\,\S(d\xi ),\\
|\langle   \mathsf C,\psi_{R,y}\rangle|
\le&
\int_{\mathbb R^d} |\widehat\psi(R\xi )|\,\S(d\xi ) = \displaystyle\int_{\mathbb{R}^{ d}}  | \hat \varphi (R\xi ) | ^{ 2} \S(d\xi )= R^{ -2d}  \textrm{Var}\left(\P(\varphi _{ R})\right) .
\end{align*}
By the spectral characterization of  hyperuniformity (\cite[Prop. 2.2]{survey-hu}), 
\begin{align*}
  \textrm{Var}\left(\P(\varphi _{ R})\right) = O(R^{ d-\alpha }),
\end{align*}
which gives 
\begin{align*}
 |\langle   \mathsf C,\psi_{R,y}\rangle| = O(R^{ -d-\alpha }).
\end{align*}
On the other hand,   \(\psi_{R,y}\) is supported by  $ \Gamma .$ The lower bound  gives
\[ 
\langle   \mathsf C,\psi_{R,y}\rangle
 \ge
K\int_{\Gamma } \|x\|^{-a}R^{-d}\psi\left(\frac{x-Ry}{R}\right)dx = 
K R^{-a}\int_{\tau _{ y}(R^{ -1}\Gamma) } |y+z|^{-a}\psi(z)\,dz
\]
and $ R^{ -1}\Gamma $ contains $ \Gamma $, hence $ \tau _{ y}(R^{ -1}\Gamma) $ supports $ \psi $
and the integral is  strictly positive  (and necessarily finite since the LHS is finite). Therefore

\[
\langle   \mathsf  C,\psi_{R,y}\rangle \gtrsim R^{-a}.
\]

Combining with the spectral upper bound, it gives $ a\geqslant d +  \alpha .$}

\subsection{Non-asymptotic results}   

Let $ \maxscale ,\delta >0$ with $$ \delta  + \maxscale(d + 2)<\frac{ 1}{d}.$$ We work    at some maximal scale $ R_{ N} = N^{\maxscale }$ and define $ \varepsilon _{ N} = N^{ -\delta }$.
To deal with boundary effects, we work in $   B (0,{ \rho _{ N}^{ -}})$  {but this time we need to define instead 
\begin{align*}
 \rho _{ N}^{ -} = \rho _{  + }N^{ 1/d}-(2\sqrt{d} + 1)N^{ 1/d-\delta }.
\end{align*} so that the margin is sufficient to accomodate for the radius of the cube $ \T[\varepsilon _{ N}N^{ 1/d}]$}. 
We can also exploit the fact that in  \cite{Serfaty-HU},  $ \beta $ is allowed to vary with $ N$, but we have to set a lower value $ \beta ^{*}>0$ to conserve the estimates at the microscopic scale.
\begin{theorem}
\label{thm:3D-Coul-finite} Let $d\in \{2,3\}, \beta ^{*}>0 $.   Let $ N\geqslant 1, R\in [1,N^{\maxscale }],\varepsilon _{ N} = N^{ -\delta },\beta = \beta  _{N}\geqslant \beta ^{*}.$
Then for   $\gamma \in (0,1], \varphi \in \F_{ \gamma }(B_{ 1})$,    
\begin{align*}
  \sup_{ x \in B(0,\rho ^{-}_{N})}   \overline{  \textrm{Var}}^{ \varepsilon _{ N}}\left(\X_{ N}^{ d,\beta  }(\tau _{ x } \varphi _{ R})\right)\leqslant C_{ \varphi }
  \begin{cases} 
  R^{ 2-\gamma }$  if $d = 2\\
  R^{ 3-\gamma }\ln(2 + R)^{ {\bf 1}\{ \gamma  = 1 \}}$  if $d = 3
   \end{cases}\\
   \end{align*}
   where $C_{ \varphi }$ does not depend on $ R,N.$
 \end{theorem}  
 
 This result therefore formally prevents the number variance in balls to have a  growth $ \geqslant R^{ d-1  + a},a>0$, outside a   set of centers $ x$ with vanishing density. See Remark~\ref{rk:local-averag} for a discussion explaining why the local averaging $ \overline{\mathrm{Var}}^{\varepsilon_N} $ should not be interpreted as a smoothing procedure.

\subsection{Proofs of theorems \ref{thm:Coul-infinite} and \ref{thm:3D-Coul-finite}}   

Let us explain how to deduce these proofs from theorems \ref{thm:non-asymp} and \ref{thm:main-asymp} and the results of  \cite{Serfaty-HU} {\rev and  \cite{ArmSer-thermal}}.\\

{\rev The first point is that the best reference measure to bound fluctuations from is not the equlibrium measure $ \mu _{N}$, but the  {\it thermal equilibrium measure} $ \mu _{ th,N}$, minimiser of a modified energy functional with an additional entropy term, see below. The bounds in  \cite{Serfaty-HU} are actually established for fluctuations  with respect to $ \mu _{ th,N}$, so we must control the discrepancy between those two  measures, as we did for Girko eigenvalues at Section \ref{sec:girko}.  In the macroscopic regime, the equilibrium probability measure $ \hat \mu_{ N} $ is more generally defined as the minimiser of the energy 
\begin{align*}
  \mathscr  E(\mu ) = \frac{ 1}{2}\displaystyle\int_{}\psi (x-y)d\mu (x)d\mu (y) + \displaystyle\int_{}V(x)d\mu (x)
\end{align*}
where here the interaction potential is $ \psi (x) = \psi _{ d}(\|x\|)$ and the confining potential is $ V(x) = \frac{ c_{ d}}{2d} \|x\|^{ 2}$. The thermal equilibrium measure $ \hat \mu _{ th,N}$  minimises 
\begin{align*}
   \mathscr  E_{ th}(\mu ): =  \mathscr  E(\mu ) + \frac{ 1}{\theta }\displaystyle\int_{}\mu \log \mu ,\text{\rm{ where }}\theta  = \beta N^{ 2/d}.
\end{align*} 
Recall that $ \beta \geqslant \beta ^{ *}$ for some fixed $ \beta ^{ *}>0.$
Under the current assumptions (Coulomb potential and quadratic confinment), there is a unique minimiser in both cases, and it has a density, still denoted $  \hat \mu _{ N}, \hat \mu _{ th,N}$ by an abuse of notation; see  \cite{serfaty-book} for a general exposition.
In the microscopic regime of the current work, we work with the rescaled versions, i.e. 
\begin{align*}
 \mu _{ N}(dx) = N \hat \mu _{ N}(N^{- 1/d}dx), \qquad \mu _{ th,N} = N \hat \mu _{th, N}(N^{ -1/d}dx).
\end{align*}
The discrepancy between the two measures is bounded in    \cite{ArmSer-thermal} (see also \cite[(2.8)]{Serfaty-HU}), for a general confinement potential $ V$. They define iteratively some functions $f_{ k},k\in \mathbb{N} $, which in the quadratic case $ V(x)  = \frac{ c_{ d}}{2d}\|x\|^{ 2}$ (of interest here) are constants because $ \Delta V =c_{ d}$ and $ \Delta (\log f_{ k}) = 0$
\begin{align*}
 f_{ k}(x)   \equiv 1.
\end{align*}
Let us give their result in our framework.
\begin{theorem}[ Armstrong, Serfaty \cite{ArmSer-thermal}]
For  $ n,m\in \mathbb{N}$ with $ 0\leqslant n\leqslant 2m-4$, for $ 0\leqslant \gamma' \leqslant \gamma \leqslant 1,$ for $ U\subset B_{ \rho _{  + }},$ if $ \theta \geqslant \theta _{ 0}(m)$, 
\begin{align*}
  | \hat \mu _{ th,N}-f_{ m-2-n/2} | _{  \mathscr  C^{ n + \gamma '}(U)} \leqslant C\theta ^{ \frac{ n + \gamma '}{2}}\exp\left(
-C\ln(\theta  \text{\rm{dist}}(U,\partial B_{ \rho _{  + }})^{ 2}
)^{ 2}\right) + C\theta ^{ 1 + n-m + \gamma '/2}.
\end{align*}
 \end{theorem}  

 We apply this   result with  $ m =  2d$ and $ n = 0,\gamma ' =  0,\gamma  = 1$, on $ U = B(0,\rho ^{ -}_{ N}N^{ -1/d})$. Note that $ \theta  \geqslant \beta ^{ *}N^{ 2/d}\geqslant \theta _{ 0}(m)$ for $ N\geqslant N_{ 0}$. Recalling  \eqref{eq:rho-plus},   $$ \text{\rm{dist}}(U,\partial B_{ \rho _{  + }}) \geqslant c(N^{ \frac{ 1}{d + 2}}\wedge N^{ \frac{ 1}{d}-\delta })N^{ \frac{- 1}{d}}\geqslant cN^{ -\delta '}$$   for some $ \delta '\in (0,\frac{ 2}{d(d + 2)})$. It gives
\begin{align}
\label{eq:bound-thermal}\notag
 \sup_{ x\in B(0,\rho ^{ -}_{ N})}  \left|
   \mu _{ th,N}(x)-1 
\right| &=   \sup_{ x\in U}  \left|
 \hat \mu _{ th,N}(x)-1
\right|  \leqslant C \exp(-C \ln(c'\beta _{ *}N^{ \frac{ 2}{d}} N^{ -2\delta '} )^{ 2}) + C (\beta _{ *}N^{ 2/d}) ^{ 1-m}\\
& = O(N^{ -2}).
\end{align}Hence   the corresponding error term will be dominated by the bound we aim to establish on the fluctuations (see below).\\

Let us now turn towards the proof of the theorems.}
Write for simplicity $  \X_{ N} = \X_{ N}^{ d,\beta }.$
Let us start by reproducing the result of  \cite{Serfaty-HU}, Corollary    2.2, originally formulated for the rescaled version $\hat\X_{ N} =  N^{ -1/d}\X_{ N},$ in a set of the form $$ \hat \Sigma  = B(0,N^{ -1/d}(\rho _{ N}^{ -}+ R_{ N})).$$

Introduce the reduced scale $\ell = N^{ -1/d}R\in [N^{ -1/d},1].$ In the notation of  \cite{Serfaty-HU}, since $ \beta \geqslant \beta ^{ *}$, $ \rho _{ \beta }\leqslant \rho _{ \beta ^{ *}}<\infty $, the results apply with $ \ell$   down from the microscopic scale $ \ell  = \rho _{ \beta ^{ *}}N^{ -1/d}$   (\cite[(2.17)]{Serfaty-HU}). Let $ f$ a fixed    $ \mathscr{C}^{ \infty }$-smooth non-negative function supported by $ B_{ 1}$, and 
\begin{align*}
 f_{ \ell}(x) = f(x/\ell)
\end{align*}
so that in  \cite{Serfaty-HU}'s notation $  | f _{ \ell}| _{ \mathscr{C}^{ k}}\leqslant C_{ f}\ell^{ -k}$ for $ k\in \mathbb{N}.$ Without loss of generality and for notational simplification we assume that $ f$ is rescaled such that in Theorem \ref{thm:non-asymp}, $ R^\circ _{ f}  =  1$.
For   $ R>0,$ $x\in \mathbb{R}^{ d},  { \rm supp}(\tau _{ x}f_{ R})\subset B(x,R)   $.
 For any rescaled function $ g_{ \ell}$ with support in $  \hat  \Sigma $,  \cite{Serfaty-HU} uses the notation\begin{align*}
\fluct (g_{ \ell}) = {\X_{ N}}(g_{ R}){\rev -\displaystyle\int_{}g_{ R}d\mu _{ th,N}}
\end{align*}
 and their Corollary 2.2 gives with $ \tau = 1$, for $ x$ such that $  { \rm supp}(\tau _{ x}f_{ \ell})\subset \hat \Sigma ,   $ uniformly in $ \beta \geqslant \beta _{ *},$
\begin{align*}
 \mathbf E \left[
\exp( |  \text{\rm{fluct}}( R^{ 2-d}\tau _{ x}f_{ \ell}) | )
\right]\leqslant C_{ f}.
\end{align*}
In particular
\begin{align*}
 \mathbf P (  R^{ 2-d} | \text{\rm{fluct}}(\tau _{ x}f_{ \ell})  | >t) = \mathbf P (\exp(R^{ 2-d}  | \text{\rm{fluct}}(\tau _{ x}f_{ \ell}) | )>e^{ t})\leqslant C_{ f}e^{ -t}.
\end{align*}
As a result, for $ y$ such that $ B( y,R)\subset    \Sigma ,$ with $ x: = N^{ -1/d}y,$ 
\begin{align}
 \label{eq:pre-var-serf} \mathbf E \left[
\fluct(\tau _{  x}f_{ \ell})^{ 2}
\right]\leqslant& cR^{ 2d-4}.\end{align}To apply Theorem \ref{thm:non-asymp}, we are interested here in the fluctuations around  Lebesgue measure, i.e. we seek for bounds on 
\begin{align*}
 \overline{\X_{ N}}(g) = \X_{ N}(g)- \Leb(g).
\end{align*}
We have, using \eqref{eq:pre-var-serf}  {\rev and  \eqref{eq:bound-thermal}},
\begin{align}\notag
      \mathbf E \left[\;\overline{\X_{ N}}(\tau _{y}f_{ R})^{ 2}
\right] \leqslant & 2 \mathbf E \left[
\fluct(\tau _{ x}f_{ \ell})^{ 2}
\right] +{\rev  2\left|
\Leb(\tau _{ y}f_{ R})-\displaystyle\int_{}\tau _{ y}f_{ R}d \mu _{ th,N}
\right|^{ 2}}\\
\label{eq:var-serf}\leqslant & CR^{ 2d-4} +{\rev  (C R^{ d}N^{ -2} )^{ 2}}=  
   \begin{cases} 
   O(1)$  if $d = 2\\
   O(R^{ 2})$  if $d = 3.
    \end{cases}
\end{align}

Note that Cauchy-Schwarz inequality gives the first order bound
\begin{align}
\label{eq:pre-fluct-exp}
 \mathbf E \left[
\X_{ N}(\tau _{ y}f_{ R})
\right]\leqslant \displaystyle\int_{}f_{ R} + \mathbf E \left[
 \overline{ \X_{ N}}(\tau _{ y}f_{ R})
\right] = O(R^{ d} + R^{ d-2}).
\end{align}

It  means that the assumptions of Theorems \ref{thm:non-asymp} and \ref{thm:main-asymp} are satisfied with $ \alpha  = 2$ in dimension $ d = 2$ and $ \alpha  = 1$ in dimension $ d = 3$.
 This proves respectively  Theorems  
\ref{thm:Coul-infinite} and    \ref{thm:3D-Coul-finite}.

 \subsection{Periodic one-dimensional Riesz gases}
 \label{sec:riesz}
 
 Let us now work on the more general model of Riesz gases. In  those models, points interact on $ \mathbb{R}^{ d}$ through the potential $  \psi _{ s}(x) = \|x \| ^{ -s}$ (and $ \psi _{ 0}(x) = -\ln(\|x\|)$), which therefore coincides with the Coulomb gas for  $ s = d-2$, up to the constant.
 We work here in dimension $d = 1, $  with $s\in (0,1),$ to build upon the results of Boursier  \cite{boursier}. We work on the one-dimensional torus $  \mathbb  T  ^{ 1}_{ N} $ with $ N$ particles. 
 The periodic Riesz potential with exponent $ s$ is   defined by
\begin{align*}
 \psi_{ N,s}(x) = \lim_{ q\to \infty }\left(
\sum_{k = -q}^{ q} \frac{ 1}{ | kN + x | ^{ s}}-\frac{ 2q^{ 1-s}}{N^{ s}(1-s)}
\right) .
\end{align*}
It is also characterized in the distributional sense with the fractional Laplacian $$ 
\begin{cases} 
 \displaystyle\int_{}\psi _{ N,s} = 0,\\ (-\Delta )^{ \frac{ 1-s}{2}}\psi _{ N,s} = c_{ s}(\delta _{ 0}-\frac{ 1}{N})
 \end{cases}$$ for some constant $ c_{ s}>0$, see  \cite{boursier} for details. Consider the random configuration $ \X_{ N}^{ s,\beta }$ with a.s. $ N$ points on $   \mathbb  T  _{ N}$ whose law is defined by   \eqref{eq:boltz}-\eqref{eq:hamil} with $ \psi  = \psi _{ N,s},C(X) \equiv 0$.  Note that  in \cite{boursier}  the macroscopic scale $ N^{ -1}\X_{ N}^{ s,\beta }$ is considered instead

  \begin{theorem}
  \label{thm:riesz}
  Let $ s\in (0,1),\beta >0.$
  \begin{itemize}
\item  Let $ \gamma \in (0,1],\varphi \in \F_{ \gamma }^{ *}$. Then there is $ C_{ \varphi }$ not depending on $ N,R$ such that
\begin{align*}
  \textrm{Var}\left(\X_{ N}^{ s,\beta }(\varphi _{ R})\right) \leqslant C_{ \varphi }V_{ 1-s,\gamma  }(R),\qquad 1\leqslant R\leqslant N.
\end{align*}
\item Any  limit point $ \P^{ s,\beta }$ of $ \X_{ N}^{ s,\beta }$ is stationary and  $ (1-s)$-hyperuniform, hence for $ \gamma >0,\varphi \in \F_{ \gamma }$, 
\begin{align*}
  \textrm{Var}\left(\P^{ s,\beta }(\varphi _{ R})\right) = O(V_{ 1-s,\gamma }(R)).
\end{align*}
\end{itemize} {\rev
 Furthermore, if $ \P^{ s,\beta }$'s reduced covariance measure $   \mathsf C$ satisfies for some $ K>0$ 
\begin{align*}
   \mathsf C(dx)\geqslant K\|x\|^{ -\eta }dx
\end{align*}(or if $ -   \mathsf C$ satisfies it)
on a semi-infinite interval, then $ \eta \geqslant 2-s$ (see the proof of Proposition \ref{prop:decorr} at Section  \ref{sec:proof-decorr}).} 
  \end{theorem}  

Theorem  \ref{thm:riesz}  shows that the transfer principle applies  without reduced scales range.
Many irregular kernels are already covered by  \cite[Theorem 2]{boursier}. Still, Theorem  \ref{thm:riesz} covers $ \F_{ \gamma }^{ *}$, which contains kernels with a singularity of the form $  | x-x_{ 0} | ^{ -a}$ with 
any   exponent $ a\in (0,1/2)$, regardless of the value of $ s$, whereas in view of  \cite[Remark 1.2]{boursier}, the conditions in  \cite{boursier} require $ a<s/2$ ($ \alpha $ in her notation); note that  \cite{boursier} was primarily interested in the CLT.

\begin{proof}
We write $ \X_{ N} = \X_{ N}^{ s,\beta }$ for the sake of simplicity.
Beware the  notation subtlety: ``$ X_{ N} $''  in  \cite{boursier} corresponds to the macroscopic scaling $   N^{ -1}\X_{ N}$ in our notation, we stick to the $ \X_{ N}$ notation here, in agreement with the rest of the paper. Since $ \X_{ N}$ is invariant under translations and has $ N$ particles in $ [0,N]$, it has intensity $ 1:$ for $ \varphi \in \B_{ c}(  \mathbb{R})$,
\begin{align*}
\mathbf E\left[
  \X_{ N}(\varphi )
\right] = \displaystyle\int_{}\varphi 
\end{align*}
(see   \eqref{eq:extend-statslins-torus} to deal with integrable functions not supported by $ [0,N]$). 

Let a non-negative function $ \xi  \in \C_{ c}^{ \infty }([-1/2,1/2])  \setminus \{0\}$  and some scale $ R>0.$
We consider the fluctuations, written in Boursier's notation, with $ \ell = R/N,$
\begin{align*}
 \overline{\X_{ N}}(\xi  _{ R}) = \X_{ N}(\xi  (./R))-R^{ d}\int_{ }\xi   = \fluct(\xi (\ell^{ -1}\cdot )).
\end{align*}In particular, 
\begin{align*}
  \textrm{Var}\left(\X_{ N}(\xi  )\right) = \mathbf E \left[\;
\overline{\X_{ N}}(\xi   )^{ 2}
\right] =  \textrm{Var}\left(\fluct(\xi (\ell^{ -1}))\right).
\end{align*}

We work with $ R\in [1,N]$, in agreement with the requirement $ \ell\in (0,1]$ from  \cite{boursier}.
The function $ \xi (\ell^{ -1}\cdot )$ satisfies   \cite[Assumption 1]{boursier}, for any exponent $ \alpha $ in her notation (note that this is unrelated to the notation $ \alpha $ in the current article, reserved for the  hyperuniformity index, we choose by default $ \alpha  = 0$ for her assumption). Therefore  \cite[Theorem 2]{boursier} gives
\begin{align}
\label{eq:var-boursier}
\mathbf E \left(\overline{\X_{ N}}(\xi  _{ R})^{ 2}\right) =   \textrm{Var}\left(\fluct(\xi (\ell^{ -1}\cdot ))\right) =  (N\ell_{\lambda })^{ s}(\sigma _{ 0}^{ 2}(\xi ) + O(\ell ^{ 2-s})) + O(N\ell )^{ 2s-1}\asymp (N\ell )^{ s} = R ^{ s},
\end{align}
the constants involved do not depend on $ \ell$, or equivalently on $ R,$
and here $ \sigma _{ 0}^{ 2}(\xi )>0.$
 Therefore we can apply Theorems \ref{thm:main-asymp} and \ref{thm:main-periodic} with the  hyperuniformity index $ \alpha  = 1-s$, and $ N = \lambda $. For Theorem \ref{thm:main-asymp}, the translation invariance of $ \X_{ N}$ yields that for $ \varphi \in \C_{ c}^0(\mathbb{R}), x\in  \mathbb  T  _{ N}, \X_{ N}(\varphi ) \stackrel{(d)}{=}  \X_{ N}(\tau _{ x}^{ N}\varphi )$, and this is easily passed to the limit in the vague topology to show that any limit point is stationary.
 
\end{proof}

By Theorem  \ref{thm:riesz}, the  spectral measure $ \S$ of any limit point $ \P^{ s,\beta }$ satisfies $\S(B_{ \varepsilon })\leqslant \varepsilon ^{ 1 + (1-s)}$ and since the variance bound  \eqref{eq:var-boursier} is two-sided this is the optimal exponent by  \cite[Prop. 2.1]{survey-hu}. It is likely that the spectral density satisfies $ \s(\xi)\sim  | \xi | ^{ 1-s}$, but we currently have no means to prove it as our estimates relate to $ \S$, not $ \s$. Yet this would imply 
\begin{align*}
 \int_{ -\varepsilon }^{ \varepsilon }\frac{ 1}{\s(\xi)}d\xi<\infty .
\end{align*}
This would corroborate by   \cite[Theorem 2]{Lr24} that the system is not-number rigid.

   \section{Spectral measure and general transfer principle}
   \label{sec:prfs}
   
   \label{sec:proof-abstract}
Recall the notation: $ \T = [-\lambda /2,\lambda /2)^{ d}$ is the $ d$-dimensional torus with volume $\lambda ^{ d}$ centered in $ 0$ with toroidal metric 
\begin{align*}
 \|x-y\|_{ \T} = \inf_{ \k\in \mathbb{Z} ^{ d}}\|\lambda \k + x-y\|
\end{align*}( where $ \|\cdot \|$ is the usual Euclidean distance) and the associated toroidal topology,  let $ \tau _{ x}^{\lambda }: \T \to \T$ the corresponding periodic translation operator.

We endow the space of non-negative Borel measures $  \mathcal  B_{\lambda }$ on $ \T$ with the weak topology, generated by the mappings $$ \mu  \mapsto  \mu (f) = \int_{ }fd\mu $$ for $ f$ in the space of continuous functions on $ \T.$ 
 A random non-negative measure $ \M$ on $ \T$ is a random element of $  \mathcal   B_{\lambda }$ for the $ \sigma $-algebra induced by the weak topology on some probability space $ (\Omega ,\mathbf P )$. We make the running assumption that $ \M$ has a finite second moment, i.e. $ \mathbf E\left[
 \M(\T)^{ 2}
\right]<\infty $, this is obviously satisfied for examples considered in this article because $ \M$ is a point process with a fixed number of points (depending on $ \lambda $). Say $ \M$ is  stationary if   $\tau _{ x}^{\lambda } \M\stackrel{(d)}{=}  \M$ for $x\in \T $. 
 A canonical way to obtain a stationary random measure is to start from a random or deterministic measure $ \M_{ 0}$ on $ \T$ and apply a translation $ \M = \tau _{ U_{\lambda }}^{\lambda }\M_{ 0}$, where $ U_{\lambda }$ is independent from $ \M_{ 0}$ and uniformly distributed on $ \T$. We actually work in the  context of  {\it weakly stationary}  $ \M$, i.e. $ \M$ is only assumed to be invariant   with respect to the first and second order marginals: for $ f\in \C^{ 0}(\T),x\in \T$, $$ \mathbf E \left[\M(
\tau _{ x}^{ \lambda }f)
\right] = \mathbf E \left[
\M(f)
\right],  \textrm{Var}\left(\M(\tau _{ x}^{ \lambda }f)\right) =  \textrm{Var}\left(\M(f)\right).$$

For $ \gamma >0$, recall that $ \F  _{  \gamma }$ is the class of functions $ \varphi\in \B_{ c}(\mathbb{R}^{ d})$ with the Fourier polynomial decay
$O(\|\xi\|^{ - \frac{ (d + \gamma )}{2} })  $. 
Typically, by  \eqref{eq:bessel}, ball indicators are in $ \F_{ 1}$, and for a function $ f\in \C_{ c}^{ \infty }(\mathbb{R}^{ d}), $  $f\in \F_{ \infty }: = \bigcap_{ \gamma >0} \mathscr  F_{\gamma  }$.
Note that for any $ \gamma >0,$ the unit square indicator is not in $  \mathscr  F_{ \gamma }$ for $ d\geqslant 2$ since, with $ \xi =  (\pi k + \pi /2,0,\dots ,0 ),k\in \mathbb{N} ,$ the optimal polynomial bound is 
\begin{align*}
 \widehat{ 1_{ [-1,1]^{ d}}}(\pi k + \pi /2,0 ,\dots ,0) =2^{ d}  \sinc(\pi k + \pi /2)\prod_{i = 2}^{ d} \sinc(0) \asymp k^{ -1} \asymp \|\xi\|^{ -1},
\end{align*}
and indeed most results of this paper do not apply to the square indicator in all generality, but they should be valid under an asymptotic covariance decay  assumption, see a discussion at \cite[Proposition 2.2]{Coste} or  \cite[Section 2.2]{survey-hu}. By  \cite[Theorem 2.16]{IoseElf}, the indicator of a compact, convex and symmetric subset $ B$ is in $\bigcup _{ \gamma >0} \F_{ \gamma }$ if its
boundary is $(d + 3)/2$-times differentiable and its principal curvatures do not vanish. 

Extend $ \M$ to functions on $ \mathbb{R}^{ d}$ through periodization: for $ \varphi \in L^{ 1}(\mathbb{R}^{ d})$, let
\begin{align*}
 \M(\varphi ) := \sum_{\k\in \mathbb{Z} ^{ d}}\int_{ \T}\varphi (\lambda \k + x)d\M(x),
\end{align*}
well defined a.s. by stationarity because, denoting by $a  =\lambda ^{ -d} \mathbf E \left[
 \M(\T)
\right]$ the intensity, Fubini identity yields
\begin{align*}
 \mathbf E\left[
 \M( | \varphi  | )
\right] = a \sum_{\k\in \mathbb{Z} ^{ d}}\int_{ \T}  | \varphi (\lambda \k + x) | dx =a  \int_{ } | \varphi  | <\infty .
\end{align*}

Define for $ R>0 $ the rescaled version $ \varphi _{ R}(x) =  \varphi (x/R),x\in \mathbb{R}^{ d}$, so that $ \M(\varphi _{ R})$ is well defined for all $ R>0$.  
\subsection{Spectral measure}

We introduce the dual Fourier space  $ \Z: = 2\pi\lambda ^{ -1}\mathbb{Z} ^{ d}$. 
A non-negative measure $ \S$ on $ \Z$ is represented through an integral $ \int_{ \Z}(...)d\S$ rather than a discrete sum. Introduce the Fourier transform on $ \T$\begin{align*}
 \what{f}{\lambda }(\xi) := \int_{ \T}f(x)e^{ -i\xi\cdot x}dx, \xi\in  \Z.
\end{align*}

 \begin{proposition}
 \label{lm:spectral}
 Let $ \M$ a weakly stationary random measure on $ \T.$
 There exists a non-negative measure $ \S$ on $ \Z$ such that
for $ f,g: \T\to  \mathbb C $ bounded measurable,
\begin{align}
\label{eq:PS}
  \textrm{Cov}\left(\M(   f),\M(g)\right) =     \int_{ \mathbb{Z} _{\lambda }}  \what{f}{\lambda }  \overline{  \what{g}{\lambda }}d\S.
\end{align}
$ \S$ is called the  {\it spectral measure} of $ \M.$ 
 \end{proposition}
Often, one renormalizes $ \S$ by $ (2\pi )^{ d}$, but in the current article it would only clutter the notation.
The proof is a classical construction, we give it in  Appendix \ref{app}.  
  We can give the explicit value of $ \S$ (not used in the following): for $ \k\in \mathbb{Z} ^{ d},\xi   \in \Z,$
\begin{align*}
 \S(\{\xi\}) =   \lambda^{ -2d}  \textrm{Var}\left(\what{\M}{ \lambda }(\xi )\right)
\end{align*}   
where $ \what{\M}{ \lambda }(\xi ) = \M(e^{ \imath \langle \xi,\cdot  \rangle})$ is the Fourier-Stieljes transform.
 For the proof of Theorem \ref{thm:abstract-general} below, we will need to understand the behavior under rescaling.
 \begin{lemma}
 \label{lm:scaling}
 For $f \in \B_{ c}^{}(\mathbb{R}^{ d}), R>0$,
\begin{align*}
 \textrm{Var}\left(\M(f_{ R})\right)=  R^{ 2d}  \int_{ \Z} | \widehat{ f }(R\xi) | ^{ 2}\S(d\xi ) .
\end{align*}
 \end{lemma}

 \begin{proof} 
 First recall  that $ \M(f) = \M( \tilde f)$ where $ \tilde f$ is defined on $ \T$ through
\begin{align*}
 \tilde f(x) = \sum_{\k\in \mathbb{Z} ^{ d}}f(x + \lambda \k).
\end{align*}

 We start from  \eqref{eq:PS}   after a $ R$-rescaling:
\begin{align*}
  \textrm{Var}\left(\M(f_{ R})\right) =  \textrm{Var}\left(\M( \widetilde{ f_{ R}})\right) =   \displaystyle\int_{\mathbb{Z} _{ \lambda }^{ d}}\left|
{\widehat{  \widetilde{  f_{ R}}}}
^{\lambda }(\xi) \right|^{ 2}
\S(d\xi ).
\end{align*}
 
For $ \xi\in  \Z$
\begin{align*}
 \what{ \widetilde{ f_{ R}}  }{\lambda } (\xi)=& \int_{  \T} \widetilde {f_{ R}}  (x)e^{ -i\xi\cdot x}dx\\
  = & \sum_{\k\in \mathbb{Z} ^{ d}}\int_{  \T} f_{ R}  (x + \lambda \k)e^{- i\xi\cdot x}dx
  \\ 
  = & \sum_{\k\in \mathbb{Z} ^{ d}}\int_{  \mathbb  T _{\lambda } ^{ d}} f_{ R}  (x + \lambda \k)e^{ -i\xi\cdot (x + \lambda \k)}dx\text{ \rm{\color{black} because}}e^{- i\xi\cdot \lambda \k} = 1
  \\ 
   = &\int_{ \mathbb{R}^{ d}}f _{ R} (y)e^{ -i\xi\cdot y}dy\\
   = &\widehat{f_{ R}  } (\xi).
\end{align*}
Therefore, on $ \Z,$   $ \what{  \widetilde{ f_{ R}}}{\lambda } = \widehat{ { f_{ R}}} =R^{ d} \hat f(R\cdot ) $ by standard considerations, which completes the proof.
  
 \end{proof}

\subsection{General transfer principle}    
For $ \zeta:\mathbb{R}_{  + }\to \mathbb{R}_{  + }$ a measurable function and $ \gamma >0$, define for $ R>0$
\begin{align*}
 \sigma _{ \zeta,\gamma  }(R) =  (d + \gamma )R^{ -\gamma }\int_{ 1}^{ R}\zeta (s)s^{ \gamma -1}ds.
\end{align*}
For instance, if $ \zeta (t) = t^{ -\alpha }$ for some $ \alpha \in \mathbb{R}$, 
\begin{align}
\label{eq:power-sigma-xi}
 \sigma _{ \zeta,\gamma  }(R) =   \begin{cases} 
 O(R^{ -\alpha }\ln(R))$  if $\alpha  = \gamma \\
 R^{ -\alpha }O(R^{ -(d + \gamma )})^{ 1-\frac{ \alpha  + d}{d + \gamma }} = O(R^{ -\gamma })$  if $\alpha >\gamma \\
 O(R^{ -\alpha })$ if $\alpha  < \gamma 
  \end{cases} = O(R^{ -d}V_{ \alpha ,\gamma }(R)).
\end{align}
Let us now give the central result of this work:
\begin{theorem}
\label{thm:abstract-general}Let $ f\in \B_{ c}^{}(\mathbb{R}^{ d})   $ non-negative such that $ |  \hat f | ^{ 2}\geqslant  \frac{ 1}{ 2}   1_{ B_{ 1\vee \sqrt{d}/2}}, \hat f(0) = 1 $.   Let $\lambda >0, \M$ a weakly stationary measure with $\mathbf E\left[
 \M( \T) 
\right]=\lambda ^{ d} $.  Let $ R_{ 0}\geqslant 1$. Assume that for  all $ R\in [1,R_{ 0}]$, 
\begin{align}
\label{eq:main-assum}
  \textrm{Var}\left(\M(f_{ R})\right)\leqslant  R^{ d}\zeta(R).
\end{align}
Then the spectral measure of $ \M$ satisfies for $ \varepsilon \in ( {1/R_{ 0}},1)$   
\begin{align*}
 \S(B_{ \varepsilon })\leqslant 2 \varepsilon ^{ d}\zeta (1/\varepsilon ).
\end{align*}
Let $ \varphi    \in \F_{ \gamma }$     for some $\gamma >0$. Then   for all $ R\in [1,R_{ 0}]$,    { $ \mathbf E\left[
 \M(\varphi _{ R})^{ 2}
\right]<\infty $ and}
\begin{align}
\label{eq:main-result}
  \textrm{Var}\left(\M(\varphi _{ R})\right) \leqslant  2    {N_{ \gamma }(\varphi )} ^{ 2} R^{ d}(\zeta (R) + c_{ d,\gamma }( 1 + \zeta (1))R^{ -\gamma } + \sigma _{ \zeta,\gamma }(R))
\end{align}
where  { $N_{ \gamma }(\varphi ) = \sup_{ \xi }(1 + \|\xi \|)^{ (d + \gamma )/2} | \hat \varphi (\xi ) | $}, $ c_{ d,\gamma }$ does not further depend on $ f,\lambda ,\M,R,\xi,\varphi  .$  
 \end{theorem}  \begin{itemize}
\item This result is mostly used for $ \zeta (t)\sim t^{ -\alpha },\alpha >0$ (hyperuniformity), giving a variance in $ V_{ \alpha ,\gamma }(R)$.
 For $ \alpha \leqslant 0,$ the bound is still true, but it concerns hyper-fluctuating systems. For $ \alpha \leqslant -d$, the bound is typically useless. 
 \item We will in general apply it to any  non-negative function $ f\in \B_{ c} (\mathbb{R}^{ d})$, up to applying rescaling and multiplicative constants in the conclusions.
\end{itemize}

\subsection{Proof of Theorem  \ref{thm:abstract-general}}

 An important preliminary  point is the translation boundedness of $ \S.$ Define the $ d$-dimensional  balls covering index
\begin{align*}
 \kappa _{ d} = \sup_{z\in \mathbb{R}^{ d}, T\geqslant 1}T^{ -d}\inf \left\{\#U:B(z,T)\subset \bigcup\limits _{ \xi\in  U}B(\xi,\sqrt{d}/2)\right\}.
\end{align*}

 \begin{lemma}
 \label{lm:trans-bound}For $ T\geqslant 1$,
$ \S(B(\xi,T))\leqslant   2\kappa _{ d}  (1 + \zeta (1))  T^{ d}$.
 \end{lemma}
 
 \begin{proof}
 Let $ \rho _{ d} = \sqrt{d}/2.$ Let  $ U =  \{\xi_{ i}^{ T};i\}\in \mathbb{R}^{ d}$ a set of  cardinality $ \leqslant \kappa _{ d}T^{ d}$ such that $ B(0,T)$ is covered by the $ B(\xi_{ i}^{ T},\rho _{ d})$. Then\begin{align*}
 |  \S(B_{ T}) |  \leqslant  |  \sum_{i}\S(B(\xi_{ i}),\rho _{ d}) | \leqslant \kappa _{ d}T^{ d}\sup_{ \xi\in  \mathbb{R}^{ d}} | \S(B(\xi,\rho _{ d})) | .
\end{align*} Recall that   $  1_{ B_{ \rho _{ d}}}\leqslant   {2}| \hat f | ^{ 2} $ on $ \mathbb{R}^{ d}.$  
By Lemma  \ref{lm:scaling} with $ R = 1$
  and Cauchy-Schwarz inequality and  \eqref{eq:PS}
\begin{align*}
  \frac{ 1}{ 2}    | \S(B(\xi,\rho _{ d} )) | \leqslant | \S( \tau _{\xi } |  \widehat{  f} | ^{ 2}) 
  =
   \textrm{Var}\left(\M(e^{ i \langle \xi,\cdot  \rangle }f)\right ).
\end{align*}
Then since $ \M$ and $ f$ are non-negative
\begin{align*}
 \textrm{Var}\left(\M(e^{ i\langle \xi,\cdot  \rangle}f)\right) \leqslant  \mathbf E\left[
  | \M(e^{ i\langle \xi,\cdot  \rangle}f) | ^{ 2}
\right]\leqslant  \mathbf E \M(f)^{ 2}\leqslant (\mathbf E \M(f))^{ 2} +   \textrm{Var}\left(\M(f)\right)
\end{align*}
 
and we assumed $$  \textrm{Var}\left(\M(f)\right)\leqslant \zeta (1)  ,\quad \mathbf E \M(f) = \int_{ }f = 1.$$   This completes the proof.
 
 \end{proof}

 \begin{proof}[Proof of Theorem \ref{thm:abstract-general}]

Therefore, by Lemma  \ref{lm:spectral},  {for $ R\in [1,R_{ 0}]$}
\begin{align}
\notag  \S(B_{ 1/R})  \leqslant  & \S(2  | \hat f(R.) | ^{ 2})  \\
\label{eq:1} =& 2\S(R^{ -2d}  | \widehat{f_{ R}} | ^{ 2}) =   2R^{ -2d}\textrm{Var}\left(\M(f_{ R})\right)  )\leqslant 2R^{ -d}\zeta (R),
\end{align}
proving the first point.
  
  Up to multiplying by a scalar, we assume $   | \hat \varphi (\xi) | \leqslant (1 + \|\xi\|)^{ -\frac{ d + \gamma }{2}}$, hence $N_{ \gamma }(\varphi ) \leqslant  1$. 
Then with $ \psi (\xi) = (R\|\xi\|)^{ -d-\gamma }$, by Lemma  \ref
 {lm:scaling},  {for $ R\in [1, R_{ 0}],$}
\begin{align*}
   \textrm{Var}\left(\M(\varphi _{ R})\right) = & R^{ 2d}\S(  | \hat \varphi (R\cdot ) | ^{ 2}) \leqslant  R^{ 2d}\S(B_{ 1/R}) + R^{ 2d}\int_{  \Z  \setminus B_{ 1/R}} \psi (\xi)\S(d\xi ).
\end{align*}
The first term is bounded by $ 2R^{ d}\zeta (R)$ by  \eqref{eq:1}. For the second term, the layer cake formula yields
\begin{align*}
 \int_{ \Z  \setminus B_{ 1/R}}\psi (\xi)\S(d\xi ) = \int_{ 0}^{ 1} \S(\{\xi\in B_{1/R}^{ c}:\psi (\xi)>t\})dt \leqslant  \int_{ 0}^{ 1}\S(B (0, t^{ -\frac{ 1}{d + \gamma }}/R))dt.
\end{align*}
We cut the  integral in two: first Lemma  \ref{lm:trans-bound} yields
\begin{align*}
\int_{ 0}^{ R^{ -(d + \gamma )}}\S(B (0, t^{ -\frac{ 1}{d + \gamma }}/R))dt\leqslant 2\kappa _{ d} (1 + \zeta (1)) \int_{ 0}^{ R^{- d - \gamma }} R^{ -d}t^{ -\frac{ d}{d + \gamma }}dt =  c_{ d,\gamma }(1 + \zeta (1))R^{ -d}(R^{ -(d + \gamma )})^{ 1-\frac{ d}{d + \gamma }}
\end{align*}
giving ultimately   $c_{ d,\gamma }(1 + \zeta (1) )R^{ d-\gamma }.$ For the second integral, the radius is $ t^{ -\frac{ 1}{d + \gamma }}/R\in [R^{ -1},1]  { \subset [R_{ 0}^{ -1},1]}$, hence with  \eqref{eq:1} the second integral  is bounded by, with $ s = Rt^{ \frac{ 1}{d + \gamma }}$,
\begin{align*}
 2 \int_{ R^{ -(d + \gamma )}}^{ 1} (t^{ -\frac{ 1}{d + \gamma }}/R)^{ d}\zeta (Rt^{ \frac{ 1}{d + \gamma }}) =&  
 2R^{ -d}\displaystyle\int_{1}^{ R}\zeta (s)s^{ \gamma -1}(d + \gamma )R^{ -\gamma }ds
 \\=& 2 R^{ - d} { \sigma_{ \zeta,\gamma } }(R),
\end{align*}
which concludes the proof.

 \end{proof}

 \subsection{Proof of Theorem  \ref{thm:non-asymp} (non-asymptotic)}  
 \label{prf:main-non-asymp}
  
  \begin{proof}
%

We will actually work  in a more general framework than in the theorem statement, because it is necessary to prove  Theorem \ref{thm:main-asymp} afterwards. \begin{itemize}
\item We assume that  the assumption  \eqref{eq:ass-asymp-stat} holds for $ R$ in the interval $ [  {\revPMP N^{ \varepsilon }},R_{ N}]$  {\revPMP  for  some $ \varepsilon \in [0,b)$ and} $ R_{ N}$  larger than some $  R^\circ _{ f}\geqslant 1$ depending on $ f$, the construction of $ R_{ f}^\circ $ is explicited in due course. Hence we do not have necessarily a range of the form $ [1,N^{ \maxscale }], R_{ N}$ is not even supposed to go to $ \infty $. 
\item  {We consider more generally any $ \gamma>0,$  $\alpha _{ 0}>\min(\alpha ,\gamma )\text{\rm{ such that }} \alpha _{ 0}\neq \gamma,\alpha _{ 0}<\min(\alpha ,\gamma ) + 1, $} and $  \delta $  satisfying $$ R_{ N}^{  d + 1 + \alpha_{ 0}} N^{ \delta }= o(N^{ 1/d}).$$ \begin{itemize}
\item For the proof of Theorem \ref{thm:non-asymp}, take $R_{ N} = N^{ \maxscale }$  with $ \maxscale(d + 2) + \delta <1/d $ and $ \gamma \leqslant 1$, but the current proof is more general. 
\item To prove Theorem \ref{thm:main-asymp} we will take $ R_{ N} = R \geqslant  R_{ f}^{ \circ}$ is constant. 
\end{itemize}
 In this more general framework, the {\it local factor} is required to satisfy 
\begin{align}
 \label{eq:local-factor}
 \varepsilon _{ N}\geqslant     R_{ N}^{ d + 1 + \alpha _{ 0}}N^{ -1/d}\end{align} 
 under the condition $ \varepsilon _{ N}\to 0.$   
\end{itemize}

 We will prove that the result  \eqref{eq:result-main-non-asymp} holds  on the scale $ [1,R_{ N}]$ with such a local factor $ \varepsilon _{ N}. $ Theorem \ref{thm:non-asymp} is then obtained with the specific choice $ R_{ N} = N^{ \maxscale },\varepsilon _{ N} = N^{ -\delta }$.\\


 {Instead of showing \eqref{eq:result-main-non-asymp} for any $ x$, we prove it for notational simplification only for $ x = 0$, with a constant not depending on the domain $ \Sigma $ (as long as $ 0\in \Sigma ^{ 2\varepsilon _{ N}N^{ 1/d}}$), this way the bound with the supremum over $ x$ can be obtained by shifting $ \Sigma $ and $ \X_{ N}$ by $ x$.}
The goal is to
 apply the periodic transfer principle of Theorem \ref{thm:abstract-general} to a stationarized version of a restriction of $ \X_{ N}$ onto $ \T$ with 
\begin{align}
 \label{eq:lambda}\lambda : = \varepsilon _{ N}N^{ 1/d}\geqslant R_{ N}^{ d + 1  + \alpha _{ 0}}.
\end{align}
Let $ U_{\lambda }$ an independent variable uniform in $ \T,$ and
\begin{align*}
  \P _{\lambda} := \tau ^{\lambda }_{ U_{\lambda }}(   \X _{ N}\cap \T),
\end{align*} see Figure  \ref{fig}. By assumption,  $ \T\subset \Sigma ^{ 2R_{ N}}.$
  By construction, $  \P _{\lambda}$ is  invariant in law under periodic translations of $ \T$.
\begin{figure}[h!] 
\begin{tikzpicture}[scale=1]

\draw[thick]
  (0,0) .. controls (2,2.5) and (8,2.5) .. (9,0)
        .. controls (8,-2) and (8,-3) .. (6,-2.5)
        .. controls (1.5,-2) and (-1,-1.5) .. (0,0);
\node at (2.5,1) {$ \Sigma$};

\def\xOx{5.2}
\def\xOy{-.5}
\def\l{1.1}
\coordinate (x0) at (\xOx,\xOy);
\draw[thick] (\xOx-\l,\xOy-\l) rectangle (\xOx + \l,\xOy + \l);
\node[right] at (\xOx-.5 ,\xOy+ \l + .2) {$ \mathbb{T}_\lambda $};

\foreach \px/\py in {
    {\xOx - 0.7*\l}/{\xOy + 0.3*\l},
    {\xOx + 0.5*\l}/{\xOy + 0.6*\l},
    {\xOx - 0.2*\l}/{\xOy - 0.5*\l},
    {\xOx + 0.7*\l}/{\xOy - 0.2*\l},
    {\xOx + 0.1*\l}/{\xOy + 0.4*\l},
    {\xOx - 0.5*\l}/{\xOy - 0.7*\l},
    {\xOx + 0.4*\l}/{\xOy - 0.6*\l},
    {\xOx - 0.3*\l}/{\xOy + 0.7*\l},
    {\xOx + 0.6*\l}/{\xOy + 0.2*\l},
    {\xOx - 0.6*\l}/{\xOy - 0.1*\l}
}{
    \filldraw[gray] (\px,\py) circle (1.5pt);
}

\coordinate (Pn) at ({\xOx + 0.7*\l}, {\xOy - 0.2*\l});
\node[gray] (PnLabel) at ({\xOx + 2.5*\l}, {\xOy - 0.8*\l}) {$ \P _{\lambda}$};
\draw[->, gray, shorten >=2pt] (Pn) -- (PnLabel);

\fill (x0) circle (2pt);
\node[below right] at (x0) {  \hspace{-.8cm}   $x_0 = 0$};

\def\Unx{0}
\def\Uny{.45}
\coordinate (Un) at (\xOx + \Unx,\xOy + \Uny);
\fill (Un) circle (2pt);
\node[right] at (Un) {$U_\lambda $};


\draw[thick] (Un) circle (0.3);
\node at (\xOx + \Unx + 3,\xOy + \Uny) {$\tau ^{ \lambda }_{   U_{\lambda }}f_R$};

\draw[-, thick] (\xOx + \Unx ,\xOy + \Uny + .3) to[bend left=40] (\xOx + \Unx + 3,\xOy + \Uny + .1);

\end{tikzpicture}
\caption{ }
\label{fig}
\end{figure}

To apply Theorem \ref{thm:abstract-general}, we must evaluate the variance over $ f_{ R}$ and then over $ \varphi _{ R}$.   We must properly link $  \textrm{Var}\left(\P_{\lambda }(g_{ R})\right)$ to   $  \textrm{Var}\left(\X_{ N}( g_{ R})\right)$ for functions $ g\in \B_{ c}^{}(B_{ 1})$ regular or irregular. From now on, fix $ R\in [  {\revPMP N^{ \varepsilon }},R_{ N}]$, and let $ C_{ f}'$ be a constant changing from line to line not depending on $ N,R,\Sigma $.

\begin{lemma}
\label{lm:fluct-var}  
For $ g \in \B_{ c}^{ }(B_{ 1})$,
\begin{align*}
 \overline{  \textrm{Var}}^{ \varepsilon _{ N}}\left(\X_{ N}(g_{ R})\right)\leqslant  \textrm{Var}\left( \P _{\lambda}(g_{ R})\right)
 +  C_{ f}'\|g\|_\infty ^{ 2}     R^{ d-\alpha _{ 0}}.
 \end{align*}
 {Let $ \alpha _{ 1} = \min(\alpha ,\alpha _{ 0}).$ }In the case $ g = f$,  we have furthermore the converse
\begin{align*}
 \textrm{Var}\left(\P_{\lambda }(f_{ R})\right)\leqslant \overline{  \textrm{Var}}^{ \varepsilon _{ N}}\left(\X_{ N}(f_{ R})\right) + C_{ f}'     R^{ d-\alpha _{ 1}}.
\end{align*}

\end{lemma}

\begin{proof}

The main issue, source of the reduced scales range in the theorem statement, lies in dealing with  $ U_{\lambda }$ being close to the boundary of $ \T$.  Let  
\begin{align*}
 \partial _{ \lambda ,R} := \{x\in \T: B(x,R)\cap \partial \T\neq \emptyset \},\qquad \mathring{\T} := \T	 \setminus \partial_{\lambda,R}.
\end{align*}
 When $ U_{\lambda }\in \partial_{\lambda,R}$,    $ \tau _{ U_{\lambda }}^{\lambda }  g_{ R}$ is possibly cut in several discontinuous pieces supported by $  \T$ (at most $ 2^{ d}$ pieces). Each such piece is supported by  a ball $ B(y,R)$ contained in $  \Sigma ^R.$ 
 \begin{lemma}
 \label{lm:brute-square}For any $ g\in \B_{ c}^{}(B_{ 1}),  {\revPMP R\in [1,R_{ N}]}$
\begin{align*}
 \sup_{ y \in \T} \mathbf E\left[
  \X _{ N}(  \tau _{ y}^{\lambda }g_{ R})^{ 2}
\right]\leqslant C_{ f}'\|g\|_\infty ^{ 2}R^{ 2d}  {\revPMP \vee N^{ 2\varepsilon d}}\\
 \sup_{ y :B(y,R)\subset \Sigma ^{ R}} \mathbf E\left[
  \X _{ N}(  \tau _{ y}g_{ R})^{ 2}
\right]\leqslant C_{ f}'\|g\|_\infty ^{ 2}R^{ 2d}{\revPMP \vee N^{ 2\varepsilon d}}.
\end{align*}
 \end{lemma}

 \begin{proof}
 Since $ f\geqslant 0$ is supported by $ B_{ 1}$ and is non-degenerate, there is $ x_{ 0}\in B_{ 1}$ 
 and some constants $ \rho _{ f}\in (0,1],c_{ f}>0$, such that $ f\geqslant c_{ f}1_{ B(x_{ 0},\rho _{ f})}$, and $$\|g\|_\infty  \tau _{ y-Rx_{ 0}}f_{ R}\geqslant \|g\|_\infty c_{ f}1_{ B(y ,\rho _{ f}R)}\geqslant c_{ f} \tau _{ y  }g_{ \rho _{ f}R}.$$  {\revPMP Let $ R' = R\vee N^{ \varepsilon }$}. Cover each ball $ B(y,R')$ by smaller balls $ B(y_{ i},\rho _{ f}R')$ (the minimal number of balls required depends on $ d$ and $ \rho _{ f}$)  ,
\begin{align*}
\sup_{ y \in \T[\lambda ] } \mathbf E \left[
 |  \X _{N}(\tau _{ y}^{\lambda }g_{ R}) | ^{ 2}
\right]\leqslant & 2^{ d}\|g\|_\infty ^{ 2}\sup_{ y:B(y,R')\subset \Sigma ^{ R'}}\mathbf E \left[
 \X_{ N}(B(y,R'))^{ 2}
\right]\\
\leqslant & C_{ f}'\|g\|_\infty ^{ 2}\sup_{ y:B(y,\rho _{ f}R')\subset  \Sigma   } \mathbf E\left[
   \X _{N}(B(y,\rho _{ f}R'))   ^{ 2}
\right]\\
\leqslant &C_{ f}'\|g\|_\infty ^{ 2}\sup_{ y\in \Sigma ^{ R'} } \mathbf E \left[
  \X_{ N}(\tau _{ y}f_{ R'})   ^{ 2}
\right]\\
\leqslant&C_{ f}'\|g\|_\infty ^{ 2} \left(\sup_{ y\in \Sigma ^{ R'}}
\mathbf E\left[\;
 \overline{\X_{ N}}(\tau _{ y}f_{ R'}) ^{ 2}
\right] + \Leb(f_{ R'})^{ 2}
\right).
\end{align*}
The first term is bounded by $ C_{ f}R'^{ d-\alpha }$ by assumption \eqref{eq:ass-asymp-stat}.
The second term is   $ R'^{ 2d}\Leb(f)^{ 2}$. It proves the first inequality of the lemma. 

The second inequality is proved similarly, starting with 
\begin{align*}
 \sup_{ y:B(y,R')\subset \Sigma ^{ R'}}\mathbf E \left[
\X_{ N}(\tau _{ y}g_{ R'})^{ 2}
\right]\leqslant \|g\|_\infty ^{ 2}\sup_{ y:B(y,R')\subset \Sigma ^{ R'}}\mathbf E \left[
\X_{ N}(B(y,R'))^{ 2}
\right].
\end{align*}
 \end{proof}

On the other hand, for interior points, i.e. for $ x\in \mathring{\T}$, $ \tau _{ x}^{\lambda }g_{ R} = \tau _{ x}g_{ R}$,  $ B(x    ,R)\subset   \T \subset    \Sigma ^{ R}$, and \eqref{eq:ass-asymp-stat} can be applied directly.
   
We are ready to prove   Lemma  \ref{lm:fluct-var}.
We start  from the conditional variance formula \begin{align}
\label{eq:cvf} 
\textrm{Var}\left(\P_{\lambda }(g_{ R})\right)-\int_{ \T[\lambda]}  \textrm{Var}\left(\X_{ N}(\tau_{ x}^{\lambda }g_{ R})\right)\frac{ dx}{\lambda^d}
  =  \textrm{Var}\left(\mathbf E\left[
 \X_{ N}(\tau ^{\lambda }_{ U_{\lambda }}g_{ R})\,|\,U_{\lambda }
\right]\right).
\end{align}
Variance positivity gives
\begin{align}
\label{eq:xxxx}
\int_{ \T[\lambda]}  \textrm{Var}\left(\X_{ N}(\tau_{ x}^{\lambda }g_{ R})\right)\frac{ dx}{\lambda^d}
\leqslant \textrm{Var}\left(\P_{\lambda }(g_{ R})\right) 
\end{align}
which almost proves the first inequality of the lemma.
 Let us deal with boundary terms. By \eqref{eq:ass-asymp-stat} and Lemma  \ref{lm:brute-square}
\begin{align*}
\left|
 \int_{ \T[\lambda]}  \textrm{Var}\left(\X_{ N}(\tau _{ x}^{\lambda }g_{ R})\right) \frac{ dx}{\lambda^d}-
\overline{\textrm{Var}}^{\varepsilon _{N}}\left(\X_{ N}( g_{ R})\right)
\right| 
= &\lambda ^{ -d} \left|
\int_{ \partial_{\lambda ,R}}  \textrm{Var}\left(\X_{ N}(\tau _{ x}^{\lambda }g_{ R})\right)dx-\int_{ \partial _{\lambda ,R}}  \textrm{Var}\left(\X_{ N}(\tau _{ x}g_{ R})\right)dx
\right|\\
\leqslant & \lambda ^{ -d}\int_{ \partial _{\lambda ,R}}\left(
 \mathbf E \left[
\X_{ N}(\tau _{ x}^{\lambda }g_{ R})^{ 2}
\right] + \mathbf E \left[
\X_{ N}(\tau _{ x}g_{ R})^{ 2} 
\right]
\right){ dx} \\
\leqslant &\lambda ^{ -d}\int_{ \partial _{\lambda ,R}}2\|g\|_\infty ^{ 2}C_{ f}'R^{ 2d}dx\\
\leqslant&C_{ f}'\frac{  \Leb(\partial _{\lambda ,R})}{\lambda^d}\|g\|_\infty ^{ 2}R^{ 2d} \leqslant C_{ f}'\|g\|_\infty ^{ 2}\frac{ R^{ 2d + 1}}{\lambda }\end{align*}
using  $ \Leb( \partial _{\lambda ,R}) = O(\lambda^{d-1}R)$.  {  By   \eqref{eq:lambda}
\begin{align}
\label{eq:bd-scale}
 \frac{ R^{ 2d + 1}}{\lambda } \leqslant R^{ d-\alpha _{ 0}}.
\end{align}}
 It proves the first inequality with   \eqref{eq:cvf}.
 
 {Let us now prove the second inequality of the lemma by treating the RHS of  \eqref{eq:cvf}. For each $ x \in \mathring\T,$ in the smooth case $ g = f,$ Assumption \eqref{eq:ass-asymp-stat} gives for $ x\in \mathring\T$, with Cauchy-Schwarz inequality
\begin{align*}
\left|
 \mathbf E\left[
 \X_{ N}(\tau _{ x}f_{ R})
\right]-\int_{ }f_{ R}
\right|\leqslant & \sqrt{\mathbf E \left[\left(
\X_{ N}(\tau _{ x}f_{ R})-\int_{ }f_{ R}
\right)^{ 2}
\right]}\leqslant C_{ f}'R^{ \frac{ d-\alpha }{2}}.
\end{align*}
Therefore with Lemma  \ref{lm:brute-square} applied to $ f$, Cauchy-Schwarz inequality yields
\begin{align*}
\left|
\mathbf E \left[
\X_{ N}(\tau _{ U_{\lambda }}^{\lambda }f_{ R})
\,|\,U_{\lambda }\right]-\int_{ }f_{ R}
\right|\leqslant & {\bf 1}\{ U_{ \lambda }\in \partial _{ \lambda ,R} \} \sup_{ x\in \partial _{\lambda ,R}}\left(
\mathbf E \left[
| \X_{ N}(\tau _{ x}^{\lambda }f_{ R} )| 
\right]  + \int_{ }f_{ R}
\right)+ C_{ f}' R^{ \frac{ d-\alpha }{2}}\\
\leqslant & {\bf 1}\{ U_{ \lambda }\in \partial _{ \lambda ,R} \}  \left(C_{ f}'\|f\|_\infty 
R^{ d} +  R^{ d}\Leb(f)
\right) + C_{ f}'R^{ \frac{ d-\alpha }{2}}\\
\textrm{Var}\left(\mathbf E \left[
\X_{ N}(\tau ^{ \lambda }_{ U_{ \lambda }}f_{ R})\,|\,U_{\lambda }
\right]\right)\leqslant  &
\mathbf E \left(\mathbf E \left[
\X_{ N}(\tau _{ U_{\lambda }}^{\lambda }f_{ R})-\int_{ }f_{ R}\,\Big|\,U_{\lambda }
\right]^{ 2}\right)\\
\leqslant & \mathbf P (U_{ \lambda }\in \partial _{ \lambda ,R})C_{ f}'R^{ 2d} + C_{ f}'R^{ d-\alpha }\\
\leqslant&  C_{ f}'\left(
\frac{ R^{ 2d + 1}}{\lambda } + R^{ d-\alpha }
\right)   
\leqslant C_{ f}'   R^{ d-\alpha _{ 1} }.
\end{align*}}

\end{proof}  

We therefore have with Lemma  \ref{lm:fluct-var} for $ g = f$, with Assumption \eqref{eq:ass-asymp-stat},
\begin{align}
\label{eq:ggg}
  \textrm{Var}\left(\P_{ \lambda }(f_{ R})\right)\leqslant C_{ f}'  R^{ d-\alpha _{ 1}}.
\end{align}

  We need another ingredient to apply Theorem \ref{thm:abstract-general}: prove that $ \mathbf E\left[
  \P _{\lambda }(\T) 
\right]= \mathbf E\left[
 \X_{ N}(\T)
\right]  \asymp   \lambda ^{ d}$. We use a similar strategy than in the lemma above.
 Let  $a: = \displaystyle\int_{}f, h = a^{ -1} f$ so that $ \int_{ }h = 1.$ Introduce the averaging operator
\begin{align*}
 \overline{   h_{ R}}(x) = \lambda ^{ -d}\displaystyle\int_{\T}h_{ R}(x + y)dy.
\end{align*} 
For $ x\in \mathring{\T}$, since $ h_{ R}$ has support in $ B(0,R)$
\begin{align*}
\overline{ h_{ R}}(x) =\lambda ^{ -d} \displaystyle\int_{\T}h_{ R}(x + y)dy = \lambda ^{ -d}\displaystyle\int_{B(-x,R)}h_{ R}(x + y)dy = \lambda ^{ -d}R^{ d}\displaystyle\int_{}h = \lambda ^{ -d}R^{ d}.
\end{align*}
Let $ \partial ^{  + }_{ \lambda ,R}$ all points at distance $ R$ or less from $ \partial \T.$
Then
\begin{align*}
 \mathbf E \left[
\X_{ N}(\overline{h_{ R}})
\right] = \mathbf E \left[
\displaystyle\int_{\partial ^{  + }_{ \lambda ,R}} \overline{h_{ R}}(x)\X_{ N}(dx)
\right] + \mathbf E \left[
\displaystyle\int_{\mathring{\T}}\lambda ^{ -d}R^{ d}\X_{ N}(dx)
\right].
\end{align*}
Let us apply this with $ R = R_{ N}.$
We have  with Cauchy-Schwarz inequality  
\begin{align*}
 \left|
\mathbf E\left[
 \P_{\lambda }(\T)
\right]-\lambda^{d}
\right| =& \left|
\mathbf E \left[
\X_{N}(\T)
\right]-\lambda^{d}
\right|\leqslant  \left|
\mathbf E \left[
\X_{N}(\T)
\right]-\frac{\lambda ^{ d}}{R_N^{ d}}\mathbf E\left[
 \X_{N}(  \overline{ h_{ R_N}} )
\right]
\right| + \left|\frac{\lambda ^{ d}}{R_N^{ d}} \mathbf E\left[\;
 \X_{N}( \overline{h_{ R_N}} )
\right]-\lambda^{d}
\right|\\
\leqslant &  \mathbf E \left[
\int_{ \partial ^{  + }_{\lambda ,R_N}}\X_{N}(dx)
\right] + \mathbf E \left[
\int_{ \partial ^{  + }_{\lambda ,R_N}}\|h\|_\infty \X_{ N}(dx)
\right]+\frac{ 1}{R_N^{ d}} \int_{ \T}\mathbf E \left|
\X_{N}(\tau _{ x}h_{ R_N})-\int_{ }h_{ R_N}
\right| dx \\
\leqslant &C_{ f}'(\|h\|_\infty  + 1)\mathbf E\left[
 \int_{ \partial ^{  + }_{\lambda ,R_N}}  \X_{ N}(B(x,1))dx
\right]
 +\frac{ 1}{R_N^{ d}}\int_{ \T}\sqrt{\mathbf E \left[\;\overline{\X_{ N}}(\tau _{ x}h_{ R_N})^{ 2}
\right]}dx\\
 \leqslant & C_{ f}' \Leb(\partial ^{  + }_{\lambda ,R_N})\sup_{ x}\sqrt{\mathbf E \left[
\X_{ N}(B(x,1))^{ 2}
\right]}+  \frac{ 1}{R_N^{ d}}\int_{ \T} {C_{ f}'R_N^{ \frac{ d-\alpha}{2} }}dx \text{\rm{ using \eqref{eq:ass-asymp-stat}}} \\
 \leqslant  & \lambda ^{ d}\left(
C'_{ f}
\frac{ R_N}{\lambda } + C_{ f}'R_N^{ -\frac{ d + \alpha }{2}}
\right).
\end{align*}
By  \eqref{eq:bd-scale}, there exists $ R^{\circ }\geqslant 1$ depending on $ f$ such that for $ R_N>R^{\circ }$, the latter quantity is bounded by $\lambda ^{ d}/2$; this is how we fix $ R^{\circ }$ throughout the proof.
There is therefore $ a_{\lambda }\in [1/2,2] $ such that $  \tilde  \P _{\lambda } :=a_{\lambda }  \P _{\lambda}$ satisfies $ \mathbf E \tilde  \P _{\lambda }(\T) =\lambda ^{ d}$.  We have with  \eqref{eq:ggg}
\begin{align}
\label{eq:xxxx}
  \textrm{Var}\left( \tilde \P_{ \lambda }(f_{ R})\right) \leqslant C_{ f}'  R^{ d-\alpha _{ 1}},
\end{align}
 {\revPMP and this holds for all $ R\in [N^{ \varepsilon },R_{ N}]$.}

  {\revPMP To prove point (i), i.e. with $ \varepsilon  = 0,$ } let $ \zeta_{ \alpha }(R): = c R^{ -\alpha }$. We have with   a computation similar to   \eqref{eq:power-sigma-xi},  { recalling $ \alpha _{ 0}\neq \gamma ,$}
\begin{align*}
 \sigma _{ \zeta,\gamma }(R)  \leqslant  C  R^{ -\min(\alpha _{ 1},\gamma )}\ln(R)^{ {\bf 1}\{ \alpha _{ 1} = \gamma  \}} .
\end{align*}  { We have $ \min(\alpha _{ 1},\gamma ) = \min(\alpha ,\gamma ).$
Distinguishing between the cases $ \alpha _{ 1} = \gamma ,\alpha _{ 1}\neq \gamma ,$ we can show 
\begin{align*}
\sigma _{ \zeta _{ \alpha _{ 1}},\gamma }(R) = R^{ -d}V_{ \alpha ,\gamma }(R).\end{align*}
 }
Theorem  \ref{thm:abstract-general} applied to $ \tilde  \P _{\lambda }$ and a rescaled version of $ f$ therefore gives the overall bound, for  $ R\in [1,R_{ N}]$
\begin{align*}
  \textrm{Var}\left( \P _{\lambda}(\varphi _{ \rho _{ f}R})\right)\leqslant  { 4\textrm{Var}\left( \tilde  \P _{\lambda }(\varphi _{ \rho _{ f}R})\right) \leqslant C_{ \varphi }C_{ f}'R^{ d}(R^{ -\alpha _{ 1}} + R^{ -\gamma }  +  \sigma _{ \zeta_{ \alpha _{ 1}}   ,\gamma } (R))\leqslant C_{ \varphi }C'_{ f}V_{ \alpha ,\gamma }(R).}
\end{align*}
Applying back Lemma  \ref{lm:fluct-var}, we can conclude the proof:
\begin{align}
\label{eq:final-main}
\overline{\textrm{Var}}^{\varepsilon _{N}}\left(\X_{ N}( \varphi _{ R})\right)
 \leqslant C_{ \varphi } C_{ f}' V_{ \alpha ,\gamma }(R).
\end{align}

 {\revPMP In the context of point (ii), 
  $ \varepsilon $ is not necessarily $ 0$. By Lemma  \ref{lm:brute-square} and Cauchy-Schwarz inequality, for $ R\in [1,N^{ \varepsilon }]$,
   positivity yields
\begin{align*}
  \textrm{Var}\left( \tilde \P_{ \lambda }(f_{ R})\right) \leqslant \mathbf E \left( \tilde \P_{ \lambda }(f_{ R})^{ 2}\right)\leqslant \mathbf E \left( \tilde \P_{ \lambda }(f_{ N^{ \varepsilon }})^{ 2}\right)\leqslant cN^{ 2d\varepsilon }.\end{align*} It means \eqref{eq:xxxx} holds for $ R\in [1,N^{ b}]$ with the RHS replaced by $R^{ d}\zeta (R)$ where 
\begin{align*}
 \zeta (s) = c
 \begin{cases} 
 N^{ 2d\varepsilon }s^{-d},s\in [1,N^{ \varepsilon }]\\
 s^{ -\alpha _{ 1}},s\in [N^{ \varepsilon },N^{ b}].
  \end{cases}
\end{align*}
To apply Theorem  \ref{thm:abstract-general} in this context, we must compute 
\begin{align*}
 \sigma _{ \zeta ,\gamma }(R) =&(d + \gamma )R^{ -\gamma }\left(
\displaystyle\int_{1}^{ N^{ \varepsilon }}N^{ 2d\varepsilon }s^{ -d}s^{ \gamma -1}ds + \displaystyle\int_{N^{ \varepsilon }}^{ R}s^{ -\alpha _{ 1}}s^{ \gamma -1}ds
\right)\\
 = &O(R^{ -\gamma }N^{ q\varepsilon } + O(R^{ -d}V_{ \alpha ,\gamma }(R))
\end{align*} where $ q  = 2d + (1\vee (\gamma -d )) $.  If $ \gamma >\alpha $, this is  bounded by $ R^{ -\alpha }$ for $ \varepsilon $ sufficiently small. Otherwise,  this is smaller than $ R^{ -(\gamma -\sqrt{\varepsilon} ) } = R^{ -d}V_{ \alpha ,\gamma -\sqrt{\varepsilon }}(R)$ for  
$ R\geqslant N^{q\sqrt{\varepsilon }}$. 
 Applying Theorem  \ref{thm:abstract-general} and  Lemma  \ref{lm:fluct-var} (applicable because  $ \sqrt{\varepsilon }q\geqslant \varepsilon $) in the same fashion as in the case $ \varepsilon  = 0$, we have the result, which concludes the proof of Theorem \ref{thm:non-asymp}.
 }

  \end{proof}
  
   \subsection{Proof of  Theorem  \ref{thm:main-asymp} (asymptotic)}  
    \label{sec:prf-asymp}
 This proof is a continuation of the previous proof.
First remark that $ \X_{ N}' = \tau _{ - x_{ N}}\X_{ N}$ satisfies the  assumptions of the theorem as well with $ \Sigma ': = \tau _{- x_{ N}} \Sigma $, hence without loss of generality we only write the proof for $ x_{ N} = 0$.  

 For notational simplification, assume up to working directly on a converging  subsequence that $ \X_{ N}\to \P$.

Since $ \varphi_{ R} $ is $ \Leb$-a.e. continuous,  for $ \Leb$-a.e. $ x\in \mathbb{R}^{ d}$, $ \tau _{ x}\varphi _{ R}$ is $ \P$-continuous, i.e. a.s. none of $ \P$'s points   lies at a discontinuity of $ \tau _{ x}\varphi _{ R}$. Therefore $ \X_{ N}(\tau _{ x}\varphi_{ R} )\to \P(\tau _{ x}\varphi _{ R})$ in law by vague convergence  by  \cite[Thm. 14.16-(iv)]{kallenberg2002foundations}.    
We apply the previous result  \eqref{eq:final-main} with $ R_{ N} = R\geqslant R^\circ $ fixed, $ \varepsilon _{ N} = L N^{ -1/d}$, satisfying  \eqref{eq:local-factor}, on the fixed range $ [1,R]$. Define 
\begin{align*}
 \rr(R,N,x) = \frac{ \textrm{Var}\left(\X_{ N}(\tau _{ x}\varphi _{ R})\right)}{V_{ \alpha ,\gamma }(R)}.
\end{align*} We therefore have  for $ L$  such that $ x_{ 0}\in \Sigma ^{ 2L}$
\begin{align*}
L^{ -d}\int_{ \T[L]} \rr(R,N,x_{ 0} + x)dx \leqslant C_{ \varphi, f}.
\end{align*}

We use the following claim, proved also later: for $ \Leb$-a.a. $ x\in \T,$
\begin{align*}
  \textrm{Var}\left(\P(\tau _{ x}\varphi _{ R})\right)\leqslant \liminf_{ N}  \textrm{Var}\left(\X_{ N}(\tau _{ x}\varphi _{ R})\right).
\end{align*}
Therefore Fatou's lemma yields
\begin{align*}
 L^{ -d}\int_{ \T[L]}\rr(R,x_{ 0} + x)dx\leqslant C_{ \varphi , f}\text{\rm{ where }} \rr(R,x) = \frac{  \textrm{Var}\left(\P(\tau _{ x}\varphi _{ R})\right)}{V_{ \alpha ,\gamma }(R)}.
\end{align*}
This concludes the proof in the general case.

In the stationary case, $ \rr(R,x)$ does not depend on $ x$, therefore it immediately gives the conclusion $$  \textrm{Var}\left(\P(\tau _{ x}\varphi _{ R})\right)\leqslant C_{ \varphi ,f} V_{ \alpha ,\gamma }(R).$$

To prove the claim, first remark that by Lemma \ref{lm:brute-square}, for $ x\in \T,$
\begin{align*}
\mathbf E \left[
\X_{ N}(\tau _{ x}\varphi _{ R})^{ 2}
\right]\leqslant C_{ f}'C_{ \varphi }R^{ 2d}.
\end{align*}
Then the claim follows from this lemma.
\begin{lemma}
\label{lm:5.6}
Let $ X_{ N}$ real random variables with $ X_{ N}  \xrightarrow[N \to \infty ]{w} X$. Assume $ \sup_{ N}\mathbf E\left[
 X_{ N}^{ 2}
\right]<\infty .$ Then 
\begin{align*}
  \textrm{Var}\left(X\right)\leqslant  \liminf_{ N}   \textrm{Var}\left(X_{ N}\right).
\end{align*}
\end{lemma}

\begin{proof}Relies on the following lemma.
\begin{lemma}
Let $ U_{N}\geqslant 0$ converging in law to $ U$. Then\begin{align*}
 \mathbf E U\leqslant \liminf_{ N} \mathbf E U_{N}.
\end{align*}
\end{lemma}

\begin{proof}
Let $ T>0$ a truncation parameter. Then $ U_{N }\wedge T\to U\wedge T.$ By convergence in law,
\begin{align*}
 \mathbf E \left[
U\wedge T
\right] \leqslant \liminf_{ N}\mathbf E \left[
U_{N}\wedge T
\right]\leqslant \liminf_{ N}\mathbf E \left[
U_{N}
\right].
\end{align*}
Then let $ T\to \infty $ and use monotone convergence theorem on the LHS.
\end{proof}

Coming back to the $ X_{ N}$. We have $ \mathbf E X^{ 2}\leqslant \liminf_{ N}\mathbf E X_{ N}^{ 2}$. We also have uniform integrability since $ \sup_{ N}\mathbf E \left[
\X_{ N}^{ 2}
\right]<\infty $, hence $ \mathbf E X_{ N}\to \mathbf E X.$

Then apply the previous lemma to $ (X_{ N}-\mathbf E X_{ N})^{ 2}  \xrightarrow[N \to \infty ]{w} (X-\mathbf E X)^{ 2}$, giving with Fatou's lemma
\begin{align*}
  \textrm{Var}\left(X\right)\leqslant  \liminf_{ N}\textrm{Var}\left(X_{ N}\right).
\end{align*}

\subsection{Proof of Theorem  \ref{thm:main-periodic} (Periodic case)
}  

(i)
 By stationarity, $ \P_{\lambda  }$ has intensity   $ \Leb$ on $ \T[\lambda]$, i.e. for $ f\in \B_{ c}^{}(\T[\lambda])$, $$ \mathbf E \left[
\P_{\lambda  }(f)
\right] =  \Leb(f).$$ 
 For $ \varphi \in \F_{ \gamma }$, the result is a direct consequence of Theorem \ref{thm:abstract-general}  in the case where $ \zeta (t) = t^{ -\alpha }$, using the computation at  \eqref{eq:power-sigma-xi}.   
By the general variance inequality $  \textrm{Var}\left(U+V\right)\leqslant  2\textrm{Var}\left(U\right)+  2\textrm{Var}\left(V\right)$, it is therefore enough to treat the case of a function $$ \varphi (x)=\|x-x_{0}\|^{-a }h (x)$$ where $ h\in \C_{ c}^{ \infty }(B_{ 1/2})$.  By translation invariance, assume without loss of generality $ x_{0}=0.$ The strategy consists in approximating $ \varphi $ with  functions from $ \F_{ \gamma }$.

Take a non-negative  mollifier $ \chi \in \C_{ c}^{ \infty }(B_{ 1/2})$ radially symmetric such that $ \int_{}\chi =1$.
  Let $ \chi _{ q}(x) = q^{ d}\chi (qx),q\geqslant 1$. 
 Let $ \varphi _{ q}  = \chi _{ q} \ast \varphi \in \B_{ c}^{}(B_{ 1})$. 
 We have\begin{align*}
 \|\varphi _{ q}\|_{L^{ 1}(\mathbb{R}^{ d})} = \|\varphi \|_{ L^{ 1}(\mathbb{R}^{ d})}  <\infty .\end{align*}Below, $ C_{ \varphi }$ is a constant depending on $ \varphi ,d,a,h,\chi $ which value might change from line to line. For $ x\in B_{ 1/2}  \setminus \{0\}$, 
\begin{align*}
 \varphi _{ q}(x) = & \int_{ B(0,1/q)} q^{ d}\chi (qy)\varphi (x + y)dy\\
 \leqslant & C_{ \varphi }\int_{ B(0,1/q)} q^{ d}\chi (qy)\|x + y\|^{ -a }dy\\
 \leqslant & C_{ \varphi }q^{ d}\displaystyle\int_{ B(0,1/q)} \|x + y\|^{ -a }dy\\
 \leqslant &  
 C_{ \varphi }q^{ d}\begin{cases} 
 \displaystyle\int_{ B(0,1/q)} (\|x\|/2)^{ -a }dy$  if $1/q<\|x\|/2\\
\displaystyle \sup_{ x_{ 0}}\displaystyle\int_{B(0,1/q)}\|x_{ 0} + y\|^{ -a}dy = \int_{ B(0,1/q)}\|y\|^{ -a }dy$   if $q\leqslant 2\|x\|^{ -1}
  \end{cases} \\
  \leqslant & C_{ \varphi } q^{ d}  
  \begin{cases} 
  \|x\|^{ -a }q^{ -d}\\
  \displaystyle\int_{0}^{ 1/q}r^{ -a  + d-1}dr
   \end{cases}\\
   \leqslant & C_{ \varphi } 
   \begin{cases} 
 \|x\|^{ -a }\\
 q^{ d}q^{ a -d}$  if $q<c\|x\|^{ -1}
    \end{cases} 
    \leqslant C_{ \varphi }\|x\|^{ -a }.
\end{align*}
We can dominate $ \varphi _{q}$ by $ C_{ \varphi }\|x\|^{ -a}$ on $ B_{ 1/2}$,   hence there is $ g$ integrable with compact support such that $  | \varphi _{ q} | \leqslant g$ a.e.. Also, $ \mathbf E \left[
\P_{\lambda  }( | g | ) 
\right] =  \int_{ } | g | <\infty $, and $ \P_{\lambda  }( | g | )<\infty $ a.s..

Clearly, $ \varphi _{ q}\to \varphi $ $ \Leb$-a.e., with the domination $ 0\leqslant \varphi _{ q}\leqslant g$, hence $ \P_{\lambda  }(\varphi _{ q})\leqslant \P_{\lambda  }(g)$ a.s.. The set $ E$ where there is no convergence is $ \Leb$-negligible, hence satisfies $ \mathbf E[ \P_{\lambda  }(E)] \leqslant    \Leb(E) = 0$. Then Lebesgue theorem  yields, by restricting to $ \T  \setminus  E$ (that has a.s. all points  from $ \P_{\lambda  }$)\begin{itemize}
\item $ \P_{\lambda  }(\varphi _{ q})\to \P_{\lambda  }(\varphi )$ a.s.  
\item $ \int_{ }\varphi _{ q}\to \int_{ }\varphi $
\item $\displaystyle\int_{}\varphi _{ q} = \mathbf E \P_{\lambda  }(\varphi _{ q})\to \mathbf E \P_{\lambda  }(\varphi )$ hence $ \mathbf E \P_{\lambda  }(\varphi ) =  \int_{ }\varphi<\infty  $
\end{itemize} 
and the same holds for $ \varphi _{ R}$  and $\varphi _{ q,R}: = \varphi _{ q}(./R)$ for fixed $ R>0$.
 
Let us prove that the $ \varphi _{ q}$ are ``uniformly'' in $ \F_{ \gamma }$.
Regarding Fourier bounds, the Fourier transform of $ \psi(x): = \|x\|^{ -a }$ in the sense of distributions on $ \mathbb{R}^{ d}$  is a function of the form $ \hat \psi (\xi) = c_{ a,d} \|\xi\|^{ a-d}$. A formal proof can be found for instance at  \cite[Thm 4.1]{SteinWeiss}, but we give a heuristic argument exploiting the  scaling property $ \psi _{ R} = R^{ a }\psi$. For a Schwartz test function $ g $, 
\begin{align*}
\langle  \widehat{  \psi _{ R}},g \rangle  = \langle \psi _{ R}, \hat g \rangle  = R^{ a } \langle \psi , \hat g  \rangle = R^{ a } \langle \hat \psi ,g  \rangle
\end{align*}
and also $ \widehat{ \psi _{ R}} = R^{ d} \hat \psi (R\cdot )$ by standard considerations on Fourier transforms of tempered distributions. We hence prove that $ \hat \psi $ is homogeneous with degree $ a-d $: 
\begin{align*}
 \hat \psi (R\;\cdot ) = R^{ a-d} \hat \psi .
\end{align*}
Since $ \psi $ is rotationally invariant, so is $ \hat \psi ,$
which justifies   the form $   c\|\xi\|^{ a-d}$.

Recall that $ \varphi =\psi   h $ with  $ h\in \C_{ c}^{ \infty }(\mathbb{R}^{ d})$, in particular $ \hat h $ is a Schwartz function and has rapid decay. Hence $  \| \hat h(\xi)\|\leqslant C_{ \varphi}(1 + \|\xi\|)^{ a-3d}$.   The convolution of Fourier distributions gives (see  \cite[Th. 7.19-c]{Rudin2})   
\begin{align*}
  |  \widehat{ \varphi  } (\xi)   | = |  \widehat{ \psi } \ast \hat h(\xi) | \leqslant C_{ \varphi }\displaystyle\int_{\mathbb{R}^{ d}}(1 + \|v\|)^{ a-3d}\|\xi + v\|^{ a-d}dv.
  \end{align*}
For $ v\in B(0,\|\xi\|/2), \|\xi  + v\|\geqslant c(1 + \|\xi\|)$, and for $ v\in B(0,\|\xi\|/2)^{ c},(1 + \|v\|)^{ a-d}\leqslant c(1 + \|\xi\|)^{a -d}, $ hence $(1 + \|v\|)^{ a-3d}\leqslant c(1 + \|\xi\|)^{a -d}(1 + \|v\|)^{ -2d}$, which yields
\begin{align*}
  | \hat \varphi (\xi) | \leqslant & C_{ \varphi}\displaystyle\int_{B(0,\|\xi\|/2)} (1 + \|\xi\|)^{ a-d}(1 + \|v\|)^{ a-3d}dv +  \displaystyle\int_{B(0,\|\xi\|/2)^{ c}}(1 + \|\xi\|)^{ a-d}(1 + \|v\|)^{ -2d}\|\xi + v\|^{ a-d}dv\\
  \leqslant &C_{ \varphi}(1 + \|\xi\|)^{ a-d} \displaystyle\int_{\mathbb{R}^{ d}}(1 + \|v\|)^{ 	a-3d}dv+C_{ \varphi} (1 + \|\xi\|)^{ a-d}\left[\displaystyle\int_{B(-\xi,1)}\|\xi + v\|^{ a-d}dv + 
\displaystyle\int_{B(-\xi,1)^{ c}}(1 + \|v\|)^{ -2d} dv
\right]\\
  \leqslant& C_{ \varphi}(1 + \|\xi\|)^{ a-d} \text{\rm{ using }}\displaystyle\int_{B(-\xi,1)}\|\xi + v\|^{ a-d}dv = \displaystyle\int_{B(0,1)}\|v\|^{ a-d}dv = c_{ d}\displaystyle\int_{0}^{ 1}r^{ d-1}r^{ a-d}dr<\infty .
\end{align*}
This also propagates to $ \varphi _{q}$ since $ \hat \chi _{q}\leqslant 1$:
\begin{align*}
 | \widehat{  \varphi _{q}}(\xi) | = |  \hat \varphi  \hat \chi _{q} | (\xi)\leqslant \displaystyle C_{ \varphi} 
(1+\|\xi\|)^{a-d} .\end{align*}
In particular the bound does not depend on $ q$, and $ \varphi _{q}\in \F_{ \gamma }$. 

We apply Theorem \ref{thm:abstract-general} to $ \P_{\lambda  }$ and a rescaled version of $ f$, noting that $ a-d = -\frac{ d + \gamma }{2}$. Since \eqref{eq:main-assum}  holds by assumption up to some constant,  \eqref{eq:main-result} also holds for $ \varphi _{ q,R}$ for $ R\in [1,\lambda^{ \maxscale }]:$
\begin{align*}
  \textrm{Var}\left(\P_{\lambda  }(\varphi _{ q,R})\right) \leqslant C_{ \varphi} V_{ \alpha ,\gamma }( R).
\end{align*}
Note that $ \varphi _{ q,R} = \chi _{ q/R} \ast \varphi _{ R}.$ Hence for fixed $ R, \varphi _{ q,R}\to \varphi _{ R}$ pointwise.
We have for each  $ R\in [1,\lambda^{ \maxscale }]$: 
\begin{align*}
  \sup_{ q} \mathbf E \left[
\P_{\lambda  }(\varphi _{ q,R})^{ 2}
\right]\leqslant C_{ \varphi}^{ 2} V_{ \alpha ,\gamma }( R) + \left(
\int_{ }\varphi _{ q,R}
\right)^{ 2}
\end{align*}
and \begin{align*}
 \displaystyle\int_{}\varphi _{ q,R} = \displaystyle\int_{}\chi _{ qR}\displaystyle\int_{}\varphi _{ R} = \displaystyle\int_{}\varphi _{ R}
\end{align*} does not depend on $ q.$
Since furthermore $ \P_{\lambda  }(\varphi _{ q,R})\to \P_{\lambda  }(\varphi _{ R})$ a.s. as $ q\to \infty $, Fatou's lemma yields 
\begin{align*}
 \mathbf E\left[
 \P_{\lambda  }(\varphi _{ R})^{ 2}
\right]\leqslant \liminf_{ q} \mathbf E\left[
 \P_{\lambda  }(\varphi _{ q,R})^{ 2}
\right]\leqslant  C_{ \varphi }^{ 2} V_{ \alpha ,\gamma }(R) + \left(
\int_{ }\varphi _{ R}
\right)^{ 2}
\end{align*}
giving the result after subtracting     :
\begin{align*}
  \textrm{Var}\left(\P_{\lambda  }(\varphi _{ R})\right)\leqslant  C_{ \varphi }^{ 2} V_{ \alpha ,\gamma }(R).
\end{align*}

{\rev
 {\bf Converse direction:}   The converse inequality in Theorem    \ref{thm:main-periodic} is in fact a contrapositive result, with roles exchanged. 
The assumption  on $ f$ yields that $ f\in \F_{ \gamma }$ for some $ \gamma >1$,  i.e. $  | \hat f(\xi) |  = O(\|\xi\|^{ -\frac{ d + \gamma }{2}})$. 
More precisely we apply again Theorem   \ref{thm:abstract-general}, but with $ ``f = 1_{ B_{ 1}},\varphi  = f, $'' on the interval $ [1,R_{ \lambda}]$.

Reason by contradiction: if  for some subsequence $ \lambda_{ j}\to \infty $ the LHS in  \eqref{eq:liminf-main} is bounded by some $ C<\infty $ whatever is the choice $ R_{ \lambda_{ j}}',$   it means that 
 for all  $ j,R\in [1,R_{ \lambda_{ j}}]$
\begin{align*}
  \textrm{Var}\left(\P_{\lambda_{ j}}(B_{ R})\right)\leqslant C V_{ \alpha ,1}(R).
\end{align*}
Note that $ V_{ \alpha ,1}(R) = R^{ d}\zeta (R)$ with  either $ \zeta (R) = R^{ -\min(\alpha,1)}$ if $ \alpha  \neq 1$ or   $ \zeta (R) = R^{ -1}\ln(R)$ if $ \alpha  = 1$. 

Theorem \ref{thm:abstract-general} implies   for all $ R\in [1,R_{ \lambda_{ j}}]$
\begin{align*}
 \textrm{Var} \left(\P_{\lambda_{ j}}(f_{ R })\right)\leqslant C'R^{ d}(R^{ -\gamma } + \zeta (R) +  \sigma _{ \zeta,\gamma }(R)) 
\end{align*}
for some finite $ C'$ not depending on $ \lambda $.\begin{itemize}
\item  \underline {First case:  $ \zeta (R) = R^{ -\min(\alpha,1) }$ for some $ \alpha\neq 1$.}  
 Then, since $ \gamma >1$,  as in \eqref{eq:power-sigma-xi}
$$
\sigma _{ \zeta,\gamma }(R) = O(R^{ -\min(\alpha,1) })  
$$ and we have for $ R\in [1,R_{ \lambda_{ j}}],$ in particular for $ R = R_{ \lambda_{ j}}$,   $  {\textrm{Var}}\left(\P_{ \lambda_{ j}}(f_{ R })\right) = O(V_{ \alpha ,1}(R))$, reaching a contradiction with	\eqref{eq:limsup}.
 \item \underline {Second case: $ \alpha  = 1$.}  Then \eqref{eq:power-sigma-xi} yields  \begin{align*}
 \sigma _{ \zeta   ,\gamma  }(R) = &(d + \gamma )R^{ -\gamma }\displaystyle\int_{1}^{ R}\ln(s)s^{ -1}s^{ \gamma -1}ds\\
  = &O(R^{ -\gamma }\ln(R)R^{ \gamma -1})\\
  =& O(R^{ -d}V_{ \alpha ,1}(R)).
\end{align*}
It also leads to a contradiction.

\end{itemize}}

\phantomsection

   \appsection{Proof of Proposition \ref{lm:spectral}  }{app}
\addcontentsline{toc}{section}{\protect\numberline{}}

  {Let the measure 
\begin{align*}
 \mu  = \mathbf E \left[
\M \otimes \M
\right]-(\mathbf E \M) \otimes (\mathbf E \M).
\end{align*}}
%
It satisfies \begin{align*}
\textrm{Cov}\left(\M(f),\M(g)\right)
 = \int_{ (\T)^{ 2}}f(x)   \overline{   g(y)} \mu (dx,dy).
\end{align*}
Let $ \mu '$ the push-forward measure of $ \mu $ under the mapping $ (x,y)  \mapsto  (x,y-x)$,  \begin{align*}
\textrm{Cov}\left(\M(f),\M(g)\right) = \int_{ ( \T)^{ 2}}f(x)  \overline{g(x+z)}\mu '(dx,dz) .
\end{align*}
   The  translation invariance of $ \M$  implies the $ x$-translation invariance of $ \mu '$, hence a disintegration theorem (\cite[Section A2.7]{DVj08}) yields a measure $ \C$ on $ \T$, called  {\it (reduced) covariance measure} such that 
\begin{align*}
\textrm{Cov}\left(\M(f),\M(g)\right) = \int_{( \T)^{ 2}}f(x)  \overline{g(x+z)}dx\C (dz).
\end{align*}
This measure  $ \C$ is positive-definite in the sense that 
\begin{align*}
 \int_{ (\T)^{ 2}} f(x)  \overline{  f(x + z)}dx\C(dz) =  \textrm{Var}\left(\M(f)\right)\geqslant 0.
\end{align*}
Therefore, by Bochner's theorem on locally compact abelian groups \cite[Th. 4.5]{Berg}, there is a non-negative measure $ \S $ on $ \Z$ such that 
\begin{align*}
\textrm{Cov}\left(\M(f),\M(g)\right)= \int_{\Z} \hat f    \overline{ \hat g}d \S, f,g\in \C^{ 0}  (\X)
\end{align*}
with the right hand side always finite. 

It remains to extend the identity from continuous test functions to bounded
 functions. Let \(\chi\) be a mollifier such that
\(\chi\geqslant 0\), \(\int_{\mathbb R^d}\chi(x)\,dx=1\), \(\chi\) has compact
support, and
\[
0\leqslant \widehat \chi \leqslant 1,
\qquad 
\widehat \chi(\xi)=\varphi(|\xi|)
\]
with \(\varphi\) non-increasing. For \(q\geqslant 1\), define
\[
\chi_q(x):=q^d\chi(qx).
\]
On the torus, \(\chi_q\) is understood as its periodization. Then\[
\widehat{\chi_q}(\xi)=\widehat \chi(\xi/q),\qquad \xi\in  \Z,
\]
and therefore
\[
0\leqslant \widehat{\chi_q}(\xi)\uparrow 1
\qquad\text{as }q\to\infty .
\]

Let \(f_q:=f*\chi_q\). Then 
\[
\|f_q\|_\infty\leqslant \|f\|_\infty,
\qquad 
f_q\to f \quad\text{in }L^1
\]
and
pointwise on \(K\setminus N_f\) where $ N_{ f}$ is $ \Leb$-negligible. By stationarity of the intensity of
\(\M\), one has
\[
\mathbb E[\M(N_f)]=0,
\]
hence \(\M(N_f)=0\) almost surely. Therefore \(f_q\to f\) for
\(\M\)-a.e. \(x\), almost surely. Since
\[
|\M(f_q)-\M(f)|
\leqslant 2\|f\|_\infty \M(\T),
\]
and \(\mathbb E[\M(\T)^2]<\infty\), dominated convergence gives
\[
\M(f_q)\to \M(f)
\qquad\text{in }L^2(\Omega).
\]
Consequently,
\[
\operatorname{Var}(\M(f_q))\to \operatorname{Var}(\M(f)).
\]

On the spectral side, for every \(q\),
\[
\operatorname{Var}(\M(f_q))
=
 \int_{\Z}|\widehat{f_q}(\xi)|^2\,\S(d\xi ).
\]
Since
\[
\widehat{f_q}(\xi)
=
\widehat f(\xi)\widehat{\chi_q}(\xi)
=
\widehat f(\xi)\widehat\chi(\xi/q),
\]
we have
\[
|\widehat{f_q}(\xi)|^2
=
|\widehat f(\xi)|^2|\widehat\chi(\xi/q)|^2
\uparrow
|\widehat f(\xi)|^2 .
\]
Thus, by monotone convergence,
\[
\lim_{q\to\infty}
 \int_{\Z}|\widehat{f_q}(\xi)|^2\,\S(d\xi )
=
 \int_{\Z}|\widehat f(\xi)|^2\,\S(d\xi ).
\]
Combining the two limits yields
\[
\operatorname{Var}(\M(f))
=
 \int_{\Z}|\widehat f(\xi)|^2\,\S(d\xi ).
\]

Finally, the covariance identity follows by polarization. Namely, for
\[
Q(h):=\operatorname{Var}(\M(h)),
\]
one has, with the convention that covariance is linear in the first argument,
\[
\operatorname{Cov}(\M(f),\M(g))
=
\frac14\sum_{k=0}^3 \mathrm i^k Q(f+\mathrm i^k g).
\]
Applying the variance identity to \(f+\mathrm i^k g\) gives
\[
\operatorname{Cov}(\M(f),\M(g))
=
\int_{\Z}\widehat f(\xi)\overline{\widehat g(\xi)}\,\S(d\xi ).
\]
 
   \section*{Acknowledgements}
   
   I am grateful to Mylène Ma\"ida for indicating to me the work of Cipolloni et al.  \cite{CipoErdosKol}.   
\end{proof} 

\phantomsection 
\addcontentsline{toc}{section}{\protect\numberline{}Bibliography}

\bibliographystyle{abbrv}


 \end{document}